\documentclass[12pt,reqno]{amsart}

\usepackage{pgfplots}
\usepackage{times}
\usepackage{amsmath,amsfonts, amstext,amssymb,amsbsy,amsopn,amsthm}
\usepackage{dsfont}
\usepackage{esint}
\usepackage{graphicx}   
\usepackage{hyperref}
\usepackage[all]{xy}
\usepackage{color}
\usepackage{tikz}
\usepackage{lineno,hyperref}

\newcommand{\lap}{\mbox{$\bigtriangleup$}}

\newcommand{\be}{\begin{equation}}
\newcommand{\ee}{\end{equation}}

\newtheorem{theorem}{Theorem}[section]
\newtheorem{lemma}[theorem]{Lemma}

\newtheorem{corollary}[theorem]{Corollary}
\theoremstyle{definition}

\theoremstyle{remark}
\newtheorem{remark}{Remark}[section]
\newtheorem{counterexample}{Counterexample}

\theoremstyle{remark}

\numberwithin{equation}{section}
\allowdisplaybreaks[4]

\begin{document}

\title{Gibbons' conjecture for entire solutions of master equations}

\author{Wenxiong Chen }
\address{Department of Mathematical Sciences, Yeshiva University, New York, NY,  10033 USA}
\email{wchen@yu.edu}

\author{Lingwei Ma}
\address{School of Mathematical Sciences, Tianjin Normal University,
Tianjin, 300387, P.R. China, and Department of Mathematical Sciences, Yeshiva University, New York, NY, 10033 USA}
\email{mlw1103@outlook.com}

\date{\today}

\begin{abstract}
In this paper, we establish a generalized version of Gibbons' conjecture in the context of the master equation
\begin{equation*}
    (\partial_t-\Delta)^s u(x,t)=f(t,u(x,t)) \,\,  \mbox{in}\,\,  \mathbb{R}^n\times\mathbb{R}.
\end{equation*}
We show that, for each $t\in\mathbb{R}$, the bounded entire solution $u(x,t)$ must be monotone increasing in one direction, and furthermore it is one-dimensional symmetric under certain uniform convergence assumption on $u$ and an appropriate decreasing condition on $f$. These conditions are slightly weaker than their counter parts proposed in the original Gibbons' conjecture.

To overcome the difficulties in proving the Gibbons' conjecture and the impediments caused by the strong correlation between space and time of fully fractional heat operator $(\partial_t-\Delta)^s$, we introduce some new ideas and provide several new insights. More precisely, we first derive a weighted average inequality, which not only provides a straightforward proof for the maximum principle in bounded domains, but also plays a crucial role in further deducing the maximum principle in unbounded domains. Such average inequality and maximum principles are essential ingredients to
carry out the sliding method, and then we apply this direct method to prove the Gibbons' conjecture in the setting of the master equation.

It is important to note that the holistic approach developed in this paper is highly versatile, and will become useful tools in investigating various qualitative properties of solutions as well as in establishing the Gibbons' conjecture for a broad range of fractional elliptic and parabolic equations and systems.

\bigskip

{\em Mathematics Subject classification} (2020): 35R11; 35K58; 35B50; 26A33.
\bigskip

{\em Keywords:} Gibbons' conjecture; master equation; average inequality, maximum principles in unbounded domains, sliding method; monotonicity; one-dimensional symmetry.
\end{abstract}

\maketitle

\section{Introduction}
\label{s:introduction}
 Gibbons' conjecture is associated with a striking question put forth by Italian mathematician Ennio De Giorgi \cite{Giorgi} during the 1970s, known as the following

\textbf{De Giorgi's conjecture.} Suppose that $u(x)$ is an entire solution of the equation
\begin{equation}\label{AC-equ}
  -\Delta u=u-u^3,\,\, x=(x',x_n)\in \mathbb{R}^n,
\end{equation}
satisfying
\begin{equation*}
  |u(x)|\leq 1\,\,\mbox{and}\,\, \frac{\partial u}{\partial x_n}>0 \,\, \mbox{in} \,\,\mathbb{R}^n.
\end{equation*}
Then the level sets of $u(x)$ must be hyperplanes, at least for $n\leq 8$.

The dimensional restrictions in De Giorgi's conjecture is related to the Bernstein problem (cf. \cite{Simons}), which asserts that the minimal graph in Euclidean space must be a hyperplane, as long as the dimension of the ambient space does not exceed $8$. Indeed, Bombieri, De Giorgi and Giusti \cite{BGG} found that the minimal graph is not a hyperplane in $\mathbb{R}^n$ with $n\geq 9$.
Despite attracting a significant amount of attention and study from numerous mathematicians for a long time, this challenging conjecture remains unproven in its full generality even today. It has only been completely resolved in $\mathbb{R}^2$ and $\mathbb{R}^3$ by Ghoussoub and Gui \cite{GG} and Amerosio and Cabre \cite{AC}, respectively.
When the dimension $4\leq n\leq8$,  Ghoussoub and Gui \cite{GG2}, and Savin \cite{Savin} confirmed this conjecture under an additional convergence assumption that
\begin{equation*}
 \displaystyle\lim_{x_n\rightarrow\pm\infty}u(x',x_n) = \pm1 \,\,\mbox{for any} \,\, x'\in \mathbb{R}^{n-1}.
\end{equation*}
 Instead, the answer is definitely negative as soon as the dimension of  $\mathbb{R}^n$ becomes greater than $8$, since Pino, Kowalczy and Wei \cite{PKW} presented counterexamples to De Giorgi's conjecture in high-dimensional spaces.

Motivated by the problem of detecting the shape of interfaces in cosmology, Gary W. Gibbons \cite{Carbou} proposed a weaker version of De Giorgi's conjecture that replaces the one-direction monotonic condition $\frac{\partial u}{\partial x_n}>0$ by a uniform convergence assumption
\begin{equation}\label{uniform-convergence}
 \displaystyle\lim_{x_n\rightarrow\pm\infty}u(x',x_n) = \pm1 \,\,\mbox{uniformly with respect to} \,\, x'\in \mathbb{R}^{n-1}.
\end{equation}
This is referred to as the Gibbons' conjecture.

\textbf{Gibbons' conjecture.} Let $u(x)$ be an entire solution of \eqref{AC-equ} satisfying
$|u(x)|\leq 1$ and \eqref{uniform-convergence}, then $u(x)$ is monotonically increasing with respect to $x_n$ and depends exclusively on $x_n$.

Fortunately, Gibbons' conjecture has been successfully established in any dimensions by various methods, including the moving plane method by Farina \cite{Farina}, the probabilistic arguments by Barlow, Bass, and Gui \cite{BBG}, and the sliding method by Berestycki, Hamel, and Monneau \cite{BHM}. It is worth mentioning that their results applied to more general nonhomogeneous terms $f$ than the De Giorgi-type nonlinearities $f(u)=u-u^3$.
Among them, the celebrated sliding method was first introduced by Berestycki and Nirenberg \cite{BN} to establish some qualitative properties of positive solutions to the local elliptic equations.

Afterwards, Wu and Chen \cite{WuChen} developed a direct sliding method, which is valuable in many applications, such as in  deriving monotonicity, one-dimensional symmetry, uniqueness, and nonexistence of solutions to elliptic equations and systems involving fractional Laplacians as well as fractional $p$-Laplacians. Please refer to \cite{CH, Liu, MZ2, MZ3} and an exhaustive survey \cite{CHM} for details. Such direct method avoids the heavy use of classical extension method  established in \cite{CS} to overcome the difficulties caused by the non-locality of fractional operators. More importantly, this direct sliding method can be used to extend and prove Gibbons' conjecture in the settings of other fractional elliptic equations involving
various nonlocal operators (cf. \cite{CBL, DQW, Le, SL, QWZ}).
In contrast, there are few papers on the Gibbons' conjecture for entire solutions of parabolic equations
besides  a recent article by Chen and Wu \cite{ChenWu}. They developed an appropriate sliding method to prove the Gibbons' conjecture for entire solutions of the following fractional reaction-diffusion equation
\begin{equation}\label{frac-para}
  u_t(x,t)+(-\Delta)^s u(x,t)=f(t,u(x,t)),\,\, \mbox{in}\,\, \mathbb{R}^n\times\mathbb{R}.
\end{equation}

Inspired by the previous literature, in this paper, we focus on the Gibbons' conjecture for the following master equation
\begin{equation}\label{model}
    (\partial_t-\Delta)^s u(x,t)=f(t,u(x,t)) ,\,\,  \mbox{in}\,\,  \mathbb{R}^n\times\mathbb{R}.
\end{equation}
We show that the entire solution of \eqref{model} is strictly monotonic in one direction and depends only on one Euclidean variable. Here the fully fractional heat operator $(\partial_t-\Delta)^s$ was first proposed by M. Riesz in \cite{Riesz}. It is a nonlocal operator of order $2s$ in space and of order $s$ in time, and can be defined as the following integral form
\begin{equation}\label{nonlocaloper}
(\partial_t-\Delta)^s u(x,t)
:=C_{n,s}\int_{-\infty}^{t}\int_{\mathbb{R}^n}
  \frac{u(x,t)-u(y,\tau)}{(t-\tau)^{\frac{n}{2}+1+s}}e^{-\frac{|x-y|^2}{4(t-\tau)}}\operatorname{d}\!y\operatorname{d}\!\tau,
\end{equation}
where the integral in $y$ is taken in the Cauchy principle value sense,
the normalization positive constant $$C_{n,s}=\frac{1}{(4\pi)^{\frac{n}{2}}|\Gamma(-s)|},$$
with $\Gamma(\cdot)$ denoting the Gamma function and $0<s<1$. Note that this operator is nonlocal both in space and time, since the value of $(\partial_t-\Delta)^s u$ at a given point $(x,t)$ depends on the values of $u$ in the whole $\mathbb{R}^n$ and on  all the past time before $t$.
The singular integral in \eqref{nonlocaloper} is well defined provided
 $$u(x,t)\in C^{2s+\epsilon,s+\epsilon}_{x,\, t,\, {\rm loc}}(\mathbb{R}^n\times\mathbb{R}) \cap \mathcal{L}(\mathbb{R}^n\times\mathbb{R})$$
for some $\varepsilon\in (0,1)$, where
the slowly increasing function space $\mathcal{L}(\mathbb{R}^n\times\mathbb{R})$ is defined by
$$ \mathcal{L}(\mathbb{R}^n\times\mathbb{R}):=\left\{u(x,t) \in L^1_{\rm loc} (\mathbb{R}^n\times\mathbb{R}) \mid \int_{-\infty}^t \int_{\mathbb{R}^n} \frac{|u(x,\tau)|e^{-\frac{|x|^2}{4(t-\tau)}}}{1+(t-\tau)^{\frac{n}{2}+1+s}}\operatorname{d}\!x\operatorname{d}\!\tau<\infty,\,\, \forall \,t\in\mathbb{R}\right\},$$
and the definition of the local parabolic H\"{o}lder space $C^{2s+\epsilon,s+\epsilon}_{x,\, t,\, {\rm loc}}(\mathbb{R}^n\times\mathbb{R})$ will be specified in Section \ref{5}\,. Particularly, if the solution $u$ is bounded, we only need to assume that $u$ is parabolic H\"{o}lder continuous to compensate the singularity of the kernel at point $(x,t)$.
What makes this problem interesting is that, when the space-time nonlocal operator $(\partial_t-\Delta)^s$ is applied to a function that only depends on either space or time, it reduces  to a familiar fractional order operator (cf. \cite{ST}). More precisely, if $u$ is only a function of $x$, then
 \begin{equation*}
   (\partial_t-\Delta)^s u(x)=(-\Delta)^s u(x),
 \end{equation*}
where $(-\Delta)^s$ is the well-known fractional Laplacian.
In recent decades, the well-posedness of solutions to elliptic equations involving the fractional Laplace operator has been extensively investigated, interested readers can refer to \cite{CLL, CLM, CLZ1, CW, CZ, DLL, LLW, LZ} and references therein.
While if $u=u(t)$, then
 \begin{equation*}
   (\partial_t-\Delta)^s u(t)=\partial_t^s u(t),
 \end{equation*}
where $\partial_t^s$ is the Marchaud fractional derivative of order $s$. Note that the fractional powers of heat operator $(\partial_t-\Delta)^s$
can be reduced to the local heat operator $\partial_t-\Delta$ as $s\rightarrow 1$ (cf. \cite{FNW}). Our main result can thus be regarded as a nonlocal generalization
of Gibbons' conjecture for the local parabolic equation in the sense that $s\rightarrow 1$.

The space-time nonlocal equation represented by \eqref{model} arises in various physical and biological phenomena, such as anomalous diffusion \cite{KBS}, chaotic dynamics \cite{Z}, biological invasions \cite{BRR} and so on.
In applications within financial field, it can also be used to model the waiting time between transactions is correlated with the ensuring price jump (cf. \cite{RSM}). From a probabilistic point of view, the master equation is fundamental in the theory of continuous time random walk, where $u$ represents the distribution of particles whose random jumps occur simultaneously with random time lag (cf. \cite{MK}). It is in contrast to the nonlocal parabolic equations like \eqref{frac-para} or dual fractional parabolic equation
\begin{equation}\label{frac-para1}
  \partial_t^s u+(-\Delta)^s u=f,
\end{equation}
where jumps are independent of the waiting times. Such strong correlation can also be reflected mathematically by observing the initial conditions for classical maximum principles of the master equation in bounded domains, as described below.

{\em Let $\Omega$ be a bounded domain in $\mathbb{R}^n$ and $[t_1, t_2]$ be an interval in $\mathbb{R}$. Assume that $u(x,t)$ is a solution of initial exterior value problem
\begin{equation}\label{IMP1}
\left\{
\begin{array}{ll}
    (\partial_t-\Delta)^su(x,t)\geq 0 ,~   &(x,t) \in  \Omega\times(t_1,t_2]  , \\
  u(x,t)\geq 0 , ~ &(x,t)  \in  (\mathbb{R}^n\setminus\Omega)\times(t_1,t_2),\\
  u(x,t)\geq 0 , ~ &(x,t)  \in \mathbb{R}^n\times(-\infty,t_1].
\end{array}
\right.
\end{equation}
Then $u(x,t)\geq 0$ in $\Omega\times(t_1,t_2]$.}

Due to the nonlocal and strongly correlated nature of the fully fractional heat operator $(\partial_t-\Delta)^s$, in order to ensure the validity of the classical maximum principle, besides the exterior condition on
$(\mathbb{R}^n\setminus\Omega)\times(t_1,t_2),$
we must also require the initial condition $u(x,t)\geq 0$ to hold on $\mathbb{R}^n\times(-\infty,t_1]$, rather than just on $\Omega\times\{t_1\}$ or on $\Omega\times(-\infty,t_1]$ as required by the maximum principle for fractional reaction-diffusion equations \eqref{frac-para} and dual fractional equation \eqref{frac-para1}, respectively. These differences can be illuminated by the following counterexample.

Let $\Omega:=(-1,1)$. For simplicity, we consider functions in separated variables form, that is,
\begin{equation*}
  u(x,t):=X(x)T(t)
\end{equation*}
on the parabolic cylinder $\Omega \times (0,1].$

Here $X\in C^{1,1}([-1,1])$ is a function of $x$ that satisfies
\begin{equation}\label{X}
X(x)\in
\left\{
\begin{array}{ll}
    [-\varepsilon,0] ,~   & \mbox{in}\,\,[-1,1], \\
   (0,1), ~ & \mbox{in}\,\,(-2,-1)\cup(1,2),
\end{array}
\right.
\mbox{and}\,\, X(x)\equiv 1\,  \mbox{in}\,\,(-\infty,-2]\cup[2,+\infty),
\end{equation}
as illustrated in Figure 1 below.
\begin{center}
\begin{tikzpicture} [scale=2]
\draw [thick,->] (-2.7,0) -- (2.7,0) node [below right] {$x$};
\draw[thick,->] (0,-0.6) -- (0,1.3) node[left] {$X(x)$};
\path node at (-0.1,0.1) {$0$};
\draw[smooth,blue,very thick]
 plot coordinates {(-2,1) (-1.85,0.97) (-1.5,0.75)(-1,0) (-0.5, -0.11) (-0.1,-0.12) (0,-0.125) (0.1,-0.12) (0.5,-0.11) (1,0) (1.5,0.75) (1.85,0.97) (2,1) };
 \draw[very thick, blue] (2,1) -- (2.7,1);
  \draw[very thick, blue] (-2,1) -- (-2.7,1);
 \draw[fill=red] (1,0) circle (0.02);
 \draw[fill=red] (-1,0) circle (0.02);
  \draw[fill=red] (2,0) circle (0.02);
   \draw[fill=red] (-2,0) circle (0.02);
        \draw[fill=red] (0,-0.125) circle (0.02);
        \draw[fill=red] (0,1) circle (0.02);
  \node at (1,-0.15) {$1$};
  \node at (2,-0.15) {$2$};
  \node at (-1,-0.15) {$-1$};
  \node at (-2,-0.15) {$-2$};
   \node at (0.12,0.85) {$1$};
       \node at (0.15,-0.225) {$-\varepsilon$};
  \draw [very thick] [dashed] [red] (-2,0)--(-2,1);
  \draw [very thick] [dashed] [red] (2,0)--(2,1);
    \draw [very thick] [dashed] [red] (-2,1)--(2,1);
\node [below=0.3cm, align=flush center,text width=16cm] at (0,-0.6)
        {Figure 1. The shape of function $X(x)$. };
\end {tikzpicture}
\end{center}
And $T\in C^{1}([0,1])$ is a function of $t$ that fulfills
\begin{equation}\label{T}
T(t)\in\left\{
\begin{array}{ll}
    (0,\varepsilon] ,~   & \mbox{in}\,\,(\frac{1}{8},\frac{7}{8}), \\
   (-1,0) ,~   & \mbox{in}\,\,(-2,-1),
\end{array}
\right.
\mbox{and}\,\,
T(t)\equiv\left\{
\begin{array}{ll}
  0 , ~ &\mbox{in}\,\,[-1,\frac{1}{8}]\cup[\frac{7}{8},1],\\
 -1 , ~ &\mbox{in}\,\,(-\infty,-2],
\end{array}
\right.
\end{equation}
as shown in the following Figure 2.
\begin{center}
\begin{tikzpicture} [scale=2]
\draw [thick,->] (-2.7,0) -- (2.7,0) node [below right] {$t$};
\draw[thick,->] (0,-1.2) -- (0,0.6) node[left] {$T(t)$};
\path node at (-0.1,-0.1) {$0$};
\draw[very thick, blue] (0.875,0) -- (1,0);
\draw[smooth,blue,very thick]
 plot coordinates {(0.125,0) (0.35, 0.06) (0.45,0.12) (0.5,0.125) (0.55,0.12) (0.65,0.06) (0.875,0)};
 \draw[very thick, blue] (-1,0) -- (0.125,0);
 \draw[smooth,blue,very thick]
 plot coordinates {(-1,0) (-1.25,-0.1) (-1.5,-0.4) (-1.75,-0.9) (-2,-1)};
 \draw[very thick, blue] (-2,-1) -- (-2.6,-1);
 \draw[fill=red] (0.125,0) circle (0.02);
  \draw[fill=red] (0.875,0) circle (0.02);
  \draw[fill=red] (1,0) circle (0.02);
    \draw[fill=red] (-1,0) circle (0.02);
        \draw[fill=red] (-2,0) circle (0.02);
        \draw[fill=red] (0,0.125) circle (0.02);
         \draw[fill=red] (0,-1) circle (0.02);
  \node at (0.125,-0.2) {$\frac{1}{8}$};
  \node at (0.875,-0.2) {$\frac{7}{8}$};
   \node at (1,-0.2) {$1$};
      \node at (-1,-0.2) {$-1$};
         \node at (-2,0.15) {$-2$};
         \node at (-0.1,0.125) {$\varepsilon$};
       \node at (0.15,-1) {$-1$};
  \draw [very thick] [dashed] [red] (0,0.125)--(0.5,0.125);
    \draw [very thick] [dashed] [red] (0,-1)--(-2,-1);
        \draw [very thick] [dashed] [red] (-2,0)--(-2,-1);
\node [below=0.3cm, align=flush center,text width=16cm] at (0,-1.1)
        {Figure 2. The shape of function $T(t)$. };
\end {tikzpicture}
\end{center}
Let $s=\frac{1}{2}$ and $\varepsilon>0$ be a sufficiently small positive constant such that
\begin{equation*}
  (\partial_t-\Delta)^{\frac{1}{2}}u(x,t)\geq 0\,\,\mbox{in}\,\, (-1,1)\times(0,1].
\end{equation*}
Please refer to Section \ref{5} for detailed calculations. The
values taken by the function $u(x,t)$ imply that
\begin{equation} \label{A100}
\left\{
\begin{array}{ll}
    (\partial_t-\Delta)^{\frac{1}{2}}u(x,t)\geq 0 ,~   &(x,t) \in  \Omega\times(0,1], \\
   u(x,t)\geq 0 , ~ &(x,t)  \in  \Omega^c\times(0,1),\\
  u(x,t)\geq 0 , ~ &(x,t)  \in \Omega\times(-\infty,0],
 \end{array}
\right.
\end{equation}
however
\begin{equation} \label{A101}
  u(x,t)\leq 0  \,\,\mbox{for}\,\,(x,t)  \in \Omega^c\times(-\infty,0],
\end{equation}
as represented in Figure 3 below.

If the master operator $(\partial_t-\Delta)^{\frac{1}{2}}$ in (\ref{A100}) is replaced by a local parabolic operator $\frac{\partial}{\partial t} -\lap$, or a nonlocal parabolic operator $\frac{\partial}{\partial t} + (-\lap)^s$, or even a dual fractional operator $\partial_t^\alpha + (-\lap)^s$,
then by the maximum principles, we must have $$u(x,t) \geq 0, \;\;\; (x,t) \in \Omega \times (0,1].$$

Nonetheless, for problem (\ref{A100}) involving the master operator, it is evident that the initial condition $u(x,t)\geq 0$ just in $\Omega\times(-\infty,0]$ does not guarantee $u(x,t)$ to be nonnegative in $\Omega\times(0,1]$. One can easily see that the function $u(x,t) = X(x)T(t)$ so constructed is negative somewhere in $\Omega\times(0,1]$. The main problem lies in (\ref{A101}), the initial condition fail to satisfied in
$$ \Omega^c\times(-\infty,0].$$

\begin{center}
\begin{tikzpicture}[scale=2]
 \draw[blue!30,fill=blue!30] (1,0) rectangle (2,1);
  \draw[blue!30,fill=blue!30] (-1,0) rectangle (-2,1);
  \draw[yellow!30,fill=yellow!30] (-1,0) rectangle (1,-1.5);
  \draw[green!30,fill=green!30] (1,0) rectangle (2,-1.5);
  \draw[green!30,fill=green!30] (-1,0) rectangle (-2,-1.5);
\draw [thick]  [black] [->, thick](-2,0)--(2,0) node [anchor=north west] {$x$};
\draw [thick]  [black!80][->, thick] (0,-1.5)--(0,1.5) node [black][ above] {$t$};
\path node at (-0.1,-0.1) {$0$};
\draw [thick] [dashed] [blue] (-2,1)--(2,1);
\draw [thick] [dashed] [blue] (1,1)--(1,-1.5);
\draw [thick] [dashed] [blue] (-1,1)--(-1,-1.5);
\draw[fill=red] (1,0) circle (0.02);
\draw[fill=red] (0,1) circle (0.02);
\draw[fill=red] (-1,0) circle (0.02);
\node at (1.1,-0.2) {$1$};
\node at (0.1,1.12) {$1$};
\node at (-0.85,-0.2) {$-1$};
\path  (0.06,0.75) [purple][semithick] node [ font=\fontsize{10}{10}\selectfont] {$(\partial_t-\Delta)^{\frac{1}{2}}u(x,t)\geq 0$,};
\path  (0,0.5) [purple][semithick] node [ font=\fontsize{10}{10}\selectfont] {where there exist points};
\path  (0.07,0.25) [purple][semithick] node [ font=\fontsize{10}{10}\selectfont]  {such that $u(x,t)<0$};
\path  (1.5,0.5) [purple][semithick] node [ font=\fontsize{10}{10}\selectfont] {$u(x,t)\geq 0$};
\path  (-1.5,0.5) [purple][semithick] node [ font=\fontsize{10}{10}\selectfont] {$u(x,t)\geq 0$};
\path  (-0.1,-0.5) [purple][semithick] node [ font=\fontsize{10}{10}\selectfont] {$u(x,t)\geq 0$};
\path  (1.5,-0.5) [purple][semithick] node [ font=\fontsize{10}{10}\selectfont] {$u(x,t)\leq 0$};
\path  (-1.5,-0.5) [purple][semithick] node [ font=\fontsize{10}{10}\selectfont] {$u(x,t)\leq 0$};
\node [below=0.5cm, align=flush center,text width=12cm] at  (0,-1.5)
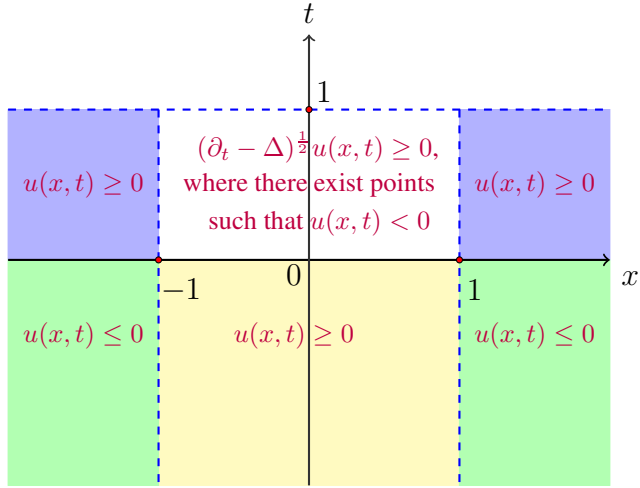
        {Figure 3. Distribution of values of function $u(x,t)=X(x)T(t)$. };
\end{tikzpicture}
\end{center}
The aforementioned counterexample shows that the initial condition on the whole $\mathbb{R}^n\times(-\infty,t_1]$ is necessary to ensure the validity of the maximum principle for parabolic equations involving the fully fractional heat operator $(\partial_t-\Delta)^s$.
Therefore, the strong correlation of master equations makes it more complicated to study compared to the parabolic equations \eqref{frac-para} and \eqref{frac-para1} that only possess nonlocal feature. For instance, please see the remark after Theorem \ref{weightAveIneq}.

The extensive practical applications highlight the significance of studying this kind of nonlocal equations in order to gain a deeper understanding of the underlying mechanisms behind various phenomena.
Substantial progress in the investigation of the existence, uniqueness and regularity of solutions to master equations has been achieved
in a series of remarkable papers \cite{ACM, CS2, ST}.
To the best of our knowledge, very little is known on the geometric behavior of solutions to master equation \eqref{model}. In the existing literature, the most common approach to studying master equation is the extension method, which extends such nonlocal equation to a local degenerate parabolic equation in a higher dimensional space. However, this method always requires cumbersome calculations and obscures the essence of the problem, and therefore may not necessarily yield the desired results. We overcome these difficulties by incorporating some new insights (which will be explained in detail later) into the sliding method to directly investigate the fully nonlocal operator $(\partial_t-\Delta)^s$. This direct method not only allows us to focus on the essential features of the nonlocal problem and avoid the complications that arise from the extension process, but also enables us to demonstrate the validity of the generalized version of Gibbons' conjecture in the context of master equation \eqref{model}, as pointed out in the following
\begin{theorem} \label{Mainresult}
Let
$$u(x,t)\in C^{2s+\epsilon,s+\epsilon}_{x,\, t,{\rm loc}}(\mathbb{R}^n\times\mathbb{R})$$
be a bounded solution of master equation
\begin{equation*}
  (\partial_t-\Delta)^su(x,t)= f(t,u(x,t)),\, \,\mbox{in}\,\, \mathbb{R}^n\times\mathbb{R},
\end{equation*}
satisfying
\begin{equation*}
\left\{\begin{array}{ll}
|u(x,t)|\leq1\,\, \mbox{for}\,\,(x,t) \in \mathbb{R}^n\times\mathbb{R},\\
\displaystyle\lim_{x_n\rightarrow\pm\infty}u(x',x_n,t) = \pm1, \,\,\mbox{uniformly for}\,\, x'=(x_1,...,x_{n-1})\in \mathbb{R}^{n-1} \,\, \mbox{and for}\,\, t\in\mathbb{R}.
\end{array}
\right.
\end{equation*}
Assume that $f(t,u)$ is continuous in $\mathbb{R}\times[-1,1]$, and for any fixed $t\in\mathbb{R}$,
\begin{equation*}
  f(t,u) \,\, \mbox{is non-increasing for} \,\, u\in[-1,-1+\delta]\cup [1-\delta,1]\,\, \mbox{with some}\,\, \delta>0.
\end{equation*}

 Then the entire solution $u(x,t)$ is strictly increasing with respect to $x_n$, and furthermore it depends only on $x_n$, that is, $$u(x',x_n,t)=u(x_n,t)$$ for any $t\in\mathbb{R}$.
\end{theorem}

\begin{remark}
Our result applies to a wide range of  more general nonlinear functions $f$, which always contains the De Giorgi-type nonlinearities $f=u-u^3$ as a special example. We would like to mention that Theorem \ref{Mainresult} is the first result that establishes the Gibbons' conjecture for master equations. This work will project new insights and perspectives into the proof of such important conjecture.
\end{remark}
Specifically, in order to effectively perform the direct sliding method, we first establish a generalized weighted average inequality for the fully fractional heat operator $(\partial_t-\Delta)^s$ as follows.
\begin{theorem}\label{weightAveIneq}
Let $$u(x,t)\in C^{2s+\epsilon,s+\epsilon}_{x,\, t,\, {\rm loc}}(\mathbb{R}^n\times\mathbb{R}) \cap \mathcal{L}(\mathbb{R}^n\times\mathbb{R}) \,.$$ If $u(x,t)$ attains its maximum at a point $(x^0,t_0)\in \mathbb{R}^n\times (-\infty,t_0]$, then there holds that
\begin{equation}\label{WAI}
  u(x^0,t_0) \leq \frac{C_0}{C_{n,s}}r^{2s}(\partial_t-\Delta)^su(x^0,t_0)+C_0r^{2s}\int_{-\infty}^{t_0-r^2}\int_{B^c_{r}(x^0)}
  \frac{u(y,\tau)e^{-\frac{|x^0-y|^2}{4(t_0-\tau)}}}{(t_0-\tau)^{\frac{n}{2}+1+s}}\operatorname{d}\!y\operatorname{d}\!\tau
\end{equation}
for any radius $r>0$, where the positive constant
$$C_0:=\frac{1}{\int_{-\infty}^{-1}\int_{B^c_{1}(0)}\frac{e^{\frac{|y|^2}{4\tau}}}{(-\tau)^{\frac{n}{2}+1+s}}\operatorname{d}\!y\operatorname{d}\!\tau},$$ and
\begin{equation}\label{measure}
  C_0r^{2s}\int_{-\infty}^{t_0-r^2}\int_{B^c_{r}(x^0)}
  \frac{e^{-\frac{|x^0-y|^2}{4(t_0-\tau)}}}{(t_0-\tau)^{\frac{n}{2}+1+s}}\operatorname{d}\!y\operatorname{d}\!\tau
=1.
\end{equation}
\end{theorem}

\begin{remark}
Due to the strong correlation of the operator $(\partial_t-\Delta)^s$, there are three major differences between this weighted average inequality and that in \cite{ChenWu}:

(i) The kernels.

(ii) Our inequality here can only be established at the points where $x$ and $t$ simultaneously reach the maximum,
as compared to the one for fractional parabolic equations \eqref{frac-para}, where it is sufficient to obtain the inequality at the maximum point with respect to $x$ for each fixed $t$ (cf. \cite{ChenWu}).

(iii) On the right hand side of the inequality, besides the integral with respect to $x$ on $B^c_{r}(x^0)$, there is another layer of integral with respect to $t$ from $-\infty$ to $t_0 -r^2$.
\end{remark}

\begin{remark}
If $u$ satisfies
\begin{equation*}
   (\partial_t-\Delta)^su(x,t)\leq 0
\end{equation*}
then at the maximum point $(x^0, t_0)$,  the key estimate established in Theorem \ref{weightAveIneq} can be simplified as follows
\begin{equation}\label{WAIS}
  u(x^0,t_0) \leq \int_{-\infty}^{t_0-r^2}\int_{B^c_{r}(x^0)}
  u(y,\tau)\operatorname{d}\!\mu_r(y,\tau),
\end{equation}
where we denote
\begin{equation*}
  \int_{-\infty}^{t_0-r^2}\int_{B^c_{r}(x^0)}
  1\operatorname{d}\!\mu_r(y,\tau):=C_0r^{2s}\int_{-\infty}^{t_0-r^2}\int_{B^c_{r}(x^0)}
  \frac{e^{-\frac{|x^0-y|^2}{4(t_0-\tau)}}}{(t_0-\tau)^{\frac{n}{2}+1+s}}\operatorname{d}\!y\operatorname{d}\!\tau
=1.
\end{equation*}
Obviously, inequality \eqref{WAIS} implies that the maximum value $u(x^0,t_0)$ can be controlled by the weighted average value of $u(x,t)$ over $B^c_{r}(x^0)\times(-\infty,t_0-r^2)$.
\end{remark}

 As a consequence of this weighted average inequality (\ref{WAIS}), we can immediately derive the maximum principle for master problems in bounded domains.
\begin{corollary}\label{MPBD}
Let $\Omega$ be a bounded domain in $\mathbb{R}^n$ and $t_1<t_2$ be two real numbers.
Suppose that
$$u(x,t)\in C^{2s+\epsilon,s+\epsilon}_{x,\, t,\,{\rm loc}}(\Omega\times(t_1,t_2])\cap \mathcal{L}(\mathbb{R}^n\times\mathbb{R})$$
is an upper semi-continuous function on $\overline{\Omega}\times[t_1,t_2]$, satisfying
\begin{equation}\label{MP1}
\left\{
\begin{array}{ll}
    (\partial_t-\Delta)^su(x,t)\leq 0 ,~   &(x,t) \in  \Omega\times(t_1,t_2]  , \\
  u(x,t)\leq 0 , ~ &(x,t)  \in \mathbb{R}^n\times(-\infty,t_1]  ,\\
  u(x,t)\leq 0 , ~ &(x,t)  \in  (\mathbb{R}^n\setminus\Omega)\times(t_1,t_2),
\end{array}
\right.
\end{equation}
then $u(x,t)\leq 0$ in $\Omega\times(t_1,t_2]$.
\end{corollary}

To prove this maximum principle, we argue by contradiction. Suppose $u$ is positive somewhere in $\Omega\times(t_1,t_2]$, then it attains its
positive maximum at some point $(x^0, t_0)$ in this parabolic cylinder. Let $r$ be sufficiently large so that $\Omega \subset B_r(x^0)$, then
applying weighted average inequality (\ref{WAIS}) and combining with the interior and exterior conditions, we arrive at $u(x^0, t_0) \leq 0$, an obvious contradiction.
\medskip

We emphasize that inequality \eqref{WAI} also plays an important role in establishing the following maximum principle in unbounded domains for master equations, which is a crucial ingredient to carry out the direct sliding method.
\begin{theorem}\label{MPUB}
Let $\Omega\subset\mathbb{R}^n$ be an open set, possibly unbounded and disconnected, and there exists a uniformly positive constant $\bar{C}$ independent of the point $x$ such that
the limit $\displaystyle\lim_{R\rightarrow +\infty}\frac{|B_R(x)\cap \Omega^c|}{|B_R(x)|}$ exists and satisfies
\begin{equation}\label{MPUB1}
  \lim_{R\rightarrow +\infty}\frac{|B_R(x)\cap \Omega^c|}{|B_R(x)|}\geq\bar{C}>0\,\,\mbox{for any}\,\, x\in\Omega.
\end{equation}
Suppose that the upper semi-continuous function (up to the boundary $\partial \Omega$)
 $$u(x,t)\in C^{2s+\epsilon,s+\epsilon}_{x,\, t,\, {\rm loc}}(\Omega\times\mathbb{R}) \cap \mathcal{L}(\mathbb{R}^n\times\mathbb{R})$$
is bounded from above in $\Omega\times\mathbb{R}$, and satisfies
\begin{equation}\label{MPUB2}
\left\{
\begin{array}{ll}
    (\partial_t-\Delta)^su(x,t)\leq 0 ,~   & \mbox{at the points in}\,\, \Omega\times\mathbb{R} \,\,\mbox{where}\,\, u(x,t)>0 , \\
  u(x,t)\leq 0 , ~ &\mbox{in}\,\, \Omega^c \times\mathbb{R}.
\end{array}
\right.
\end{equation}
Then there holds that
\begin{equation}\label{MPUB3}
u(x,t)\leq0 \,\,\mbox{in} \,\,  \Omega\times\mathbb{R}.
\end{equation}
\end{theorem}
\begin{remark}
Roughly speaking, condition \eqref{MPUB1} indicates that the ``size'' of the complement of $\Omega$ is ``not too small" as compared to the ``size'' of $\Omega$ as measured by limit of the ratio. For instance, this condition is satisfied when $\Omega$ is a half space, a stripe, an Archimedean spiral and so on. However, condition \eqref{MPUB1} is not fulfilled if,  for example
$$\Omega:=\{x=(x',x_n)\mid x'\in\mathbb{R}^{n-1},\,x_n>1 \,\mbox{or}\, x_n<-1\}.$$
Although its complement $\Omega^c$ is an infinite slab  with infinite measure, it is still
''much too small'' as compared with the ``size'' of $\Omega$ in the sense of limit defined in condition \eqref{MPUB1}.  More precisely, in this case, we have
$$\lim_{R\rightarrow +\infty}\frac{|B_R(x)\cap\Omega^c|}{|B_R(x)|} = 0.$$
\end{remark}

\begin{remark}
Note that the condition ``$u$ is bounded from above" is necessary to guarantee the validity of the above maximum principle, as shown in the following counterexample.

For simplicity, we consider functions of $x$ only. Let $$u(x,t)=u(x):=(x_n)_+^s$$ for $x\in \mathbb{R}^n$, then it is well known that $u(x)$ is a solution of the problem
\begin{equation*}
\left\{\begin{array}{r@{\ \ }c@{\ \ }ll}
&&\left(\partial_t-\Delta\right)^{s}u(x)=\left(-\Delta\right)^{s}u(x)=0, & \ \ x\in\mathbb{R}^n_+\,, \\[0.05cm]
&&u(x)= 0, & \ \ x\in\mathbb{R}^n\setminus\mathbb{R}^n_+\,.
\end{array}\right.
\end{equation*}
However, it is evident that $u(x) > 0$ in $\mathbb{R}^n_+$, which violates the conclusion of Theorem \ref{MPUB} with $\Omega=\mathbb{R}^n_+$.
\end{remark}

In addition to performing the sliding method to prove the Gibbons' conjecture  stated in Theorem \ref{Mainresult}\,, the aforementioned maximum principle can also be directly applied to establish the monotonicity of solutions to master equations on an upper half space.
\begin{corollary}\label{Application}
Let
$\mathbb{R}^n_+:=\{x\in\mathbb{R}^n\mid x_n>0\}$
be an upper half space, and
$$u(x,t)\in C^{2s+\epsilon,s+\epsilon}_{x,\, t,\, {\rm loc}}(\mathbb{R}^n_+\times\mathbb{R})$$
be a bounded solution of
\begin{equation}\label{AP1}
\left\{
\begin{array}{ll}
    (\partial_t-\Delta)^su(x,t)= f(t,u(x,t)),~   & \mbox{in}\,\, \mathbb{R}^n_+\times\mathbb{R}, \\
  u(x,t)> 0 , ~ &\mbox{in}\,\, \mathbb{R}^n_+\times\mathbb{R},\\
  u(x,t)= 0 , ~ &\mbox{in}\,\, (\mathbb{R}^n\setminus\mathbb{R}^n_+)\times\mathbb{R},
\end{array}
\right.
\end{equation}
where $u$ is continuous up to the boundary $\partial\mathbb{R}^n_+$, and the nonhomogeneous term $f(t,u)$ is monotonically decreasing with respect to $u$. Then $u(x,t)$ is strictly increasing with respect to $x_n$ in $\mathbb{R}^n_+$ for any $t\in\mathbb{R}$\,.
\end{corollary}
It is significant to mention that the holistic approach developed in this paper is very general and can be applied to investigate various qualitative properties of solutions for a wide range of fractional elliptic and parabolic equations and systems.

The remaining part of this paper will proceed as follows.
Section \ref{5} consists of the definition of parabolic H\"{o}lder space, the detailed calculation of the above counterexample and some frequently used estimates in what follows.
In Section \ref{2}\,, we derive a weighted average inequality applicable to the fully fractional heat operator $(\partial_t-\Delta)^s$, which is a key estimate for establishing the subsequent results. On this basis, Section \ref{3} is devoted to obtaining the maximum principle in unbounded domains and its straightforward applications.
Such a maximum principle plays an essential role in implementing the sliding method. Incorporating the aforementioned average inequality and the maximum principle into the sliding method, we complete the proof of Gibbons' conjecture for master equation in the last section.

\section{Preliminaries}\label{5}

In this section, we collect definitions and derive auxiliary results that are needed in establishing our main theorems. Throughout this paper, $C$ will denote a positive constant whose value may be different from line to line.

We start by providing the definition of parabolic H\"{o}lder space
$$C^{2\alpha,\alpha}_{x,\, t}(\mathbb{R}^n\times\mathbb{R}),$$
which plays an essential role in ensuring  that the fully fractional heat operator $(\partial_t-\Delta)^s$ is well-defined (cf. \cite{Kry}). More precisely,
\begin{itemize}
\item[(i)]
When $0<\alpha\leq\frac{1}{2}$, if
$u(x,t)\in C^{2\alpha,\alpha}_{x,\, t}(\mathbb{R}^n\times\mathbb{R})$, then there exists a constant $C>0$ such that
\begin{equation*}
  |u(x,t)-u(y,\tau)|\leq C\left(|x-y|+|t-\tau|^{\frac{1}{2}}\right)^{2\alpha}
\end{equation*}
for any $x,\,y\in\mathbb{R}^n$ and $t,\,\tau\in \mathbb{R}$.
\item[(ii)]
When $\frac{1}{2}<\alpha\leq1$, we say that
$$u(x,t)\in C^{2\alpha,\alpha}_{x,\, t}(\mathbb{R}^n\times\mathbb{R}):=C^{1+(2\alpha-1),\alpha}_{x,\, t}(\mathbb{R}^n\times\mathbb{R}),$$ if $u$ is $\alpha$-H\"{o}lder continuous in $t$ uniformly with respect to $x$ and its gradient $\nabla_xu$ is $(2\alpha-1)$-H\"{o}lder continuous in $x$ uniformly with respect to $t$ and $(\alpha-\frac{1}{2})$-H\"{o}lder continuous in $t$ uniformly with respect to $x$.
\item[(iii)] While for $\alpha>1$, if
$u(x,t)\in C^{2\alpha,\alpha}_{x,\, t}(\mathbb{R}^n\times\mathbb{R}),$
then it means that
$$\partial_tu,\, D^2_xu \in C^{2\alpha-2,\alpha-1}_{x,\, t}(\mathbb{R}^n\times\mathbb{R}).$$
\end{itemize}
In addition, we can analogously define the local parabolic H\"{o}lder space $C^{2\alpha,\alpha}_{x,\, t,\, \rm{loc}}(\mathbb{R}^n\times\mathbb{R})$.

Next, we present a detailed calculation of the counterexample mentioned in the introduction regarding the maximum principle of master equation not being valid when the initial condition does not satisfy nonnegativity on the whole $\mathbb{R}^n\times(-\infty,t_1]$.
\begin{counterexample}
Let $u(x,t)=X(x)T(t)$, where $X(x)\in C^{1,1}([-1,1])$ and $T(t)\in C^{1}([0,1])$ are bounded functions defined in \eqref{X} and \eqref{T}, respectively. Then there exists
a sufficiently small constant $\varepsilon\in (0,1)$ such that
\begin{equation*}
  (\partial_t-\Delta)^{\frac{1}{2}}u(x,t)\geq 0\,\,\mbox{in}\,\, (-1,1)\times(0,1].
\end{equation*}
\end{counterexample}
\begin{proof}
For $(x,t)\in (-1,1)\times(0,1]$, applying the definitions of $(\partial_t-\Delta)^{\frac{1}{2}}$, we divide the integral domain into three parts
\begin{eqnarray}\label{example}
&&(\partial_t-\Delta)^{\frac{1}{2}} u(x,t)\nonumber\\
 &=&C_{1,\frac{1}{2}}\int_{-\infty}^{t}\int_{-\infty}^{\infty}
  \frac{X(x)T(t)-X(y)T(\tau)}{(t-\tau)^{2}}e^{-\frac{|x-y|^2}{4(t-\tau)}}\operatorname{d}\!y\operatorname{d}\!\tau \nonumber\\
  &=& C_{1,\frac{1}{2}}\left(\int_{-\infty}^{0}\int_{|y|> 1}+\int_{-\infty}^{0}\int_{-1}^{1}+\int_{0}^{t}\int_{-\infty}^\infty
  \frac{X(x)T(t)-X(y)T(\tau)}{(t-\tau)^{2}}e^{-\frac{|x-y|^2}{4(t-\tau)}}\operatorname{d}\!y\operatorname{d}\!\tau \right)\nonumber\\
&  =:& I+II+III.
\end{eqnarray}
Now we are going to estimate each of these three integrals separately. According to the definition of function $T$, there is no need to worry about the singularity when $\tau\in (-\infty,0)$ is close to $t\in(0,1)$, since $X(x)T(t)-X(y)T(\tau)=0$ as $\tau\rightarrow 0^-$ and $t\rightarrow 0^+$. Then
in terms of \eqref{X}, \eqref{T} and the small constant $\varepsilon\in (0,1)$, we first estimate $I$ and $II$ as follows
\begin{eqnarray}\label{I}
  I &\geq& C\int_{-\infty}^{2}\int_{|y|>2}
  \frac{-\varepsilon^2+1}{(t-\tau)^{2}}e^{-\frac{|x-y|^2}{4(t-\tau)}}\operatorname{d}\!y\operatorname{d}\!\tau -C\varepsilon^2 \nonumber\\
   &\geq& C(1-\varepsilon^2)\geq C_0>0,
\end{eqnarray}
and
\begin{equation}\label{II}
  |II|\leq C(\varepsilon+\varepsilon^2)\leq C\varepsilon.
\end{equation}
Due to the presence of singular point $(x,t)$, the estimate of $III$ is somewhat complicated, and we need to divide it into the following two parts
\begin{eqnarray*}
  III &=& C_{1,\frac{1}{2}}\left(\int_{0}^{t}\int_{|y-x|\geq \sqrt{\varepsilon}}+\int_{0}^{t}\int_{|y-x|< \sqrt{\varepsilon}}
  \frac{X(x)T(t)-X(y)T(\tau)}{(t-\tau)^{2}}e^{-\frac{|x-y|^2}{4(t-\tau)}}\operatorname{d}\!y\operatorname{d}\!\tau \right)  \\
   &=:&  III_1+III_2.
\end{eqnarray*}
With respect to the estimate of $III_1$, we directly compute
\begin{eqnarray*}
  |III_1|&\leq& C(\varepsilon^2+\varepsilon)\int_{0}^{t}\int_{|y-x|\geq \sqrt{\varepsilon}} \frac{1}{(t-\tau)^{2}}e^{-\frac{|x-y|^2}{4(t-\tau)}}\operatorname{d}\!y\operatorname{d}\!\tau \\
  &=& -C(\varepsilon^2+\varepsilon)\int_{|y-x|\geq \sqrt{\varepsilon}} \int_{0}^{t}\frac{\operatorname{d}}{\operatorname{d}\!\tau} \left(e^{-\frac{|x-y|^2}{4(t-\tau)}}\right)\frac{4}{|x-y|^2}\operatorname{d}\!y\\
  &\leq&C(\varepsilon^2+\varepsilon)\int_{\sqrt{\varepsilon}}^{+\infty}\frac{\operatorname{d}\!r}{r^2}
  \leq C\sqrt{\varepsilon}
\end{eqnarray*}
The estimate of $III_2$ proceeds via a change of variables, Taylor expansion, and the definition of the Cauchy principal value, which yields
\begin{eqnarray*}
  |III_2| &=& C_{1,\frac{1}{2}}\left|\int_{0}^{t}\int_{|y-x|< \sqrt{\varepsilon}}
  \frac{X(x)(T(t)-T(\tau))+(X(x)-X(y))T(\tau)}{(t-\tau)^{2}}e^{-\frac{|x-y|^2}{4(t-\tau)}}\operatorname{d}\!y\operatorname{d}\!\tau\right|  \\
   &\leq& C\varepsilon\int_0^{t}\frac{1}{(t-\tau)^{\frac{1}{2}}} \operatorname{d}\!\tau  +C_{1,\frac{1}{2}}\left|\int_{0}^{t}\int_{|y-x|< \sqrt{\varepsilon}}
  \frac{O(|x-y|^2)T(\tau)}{(t-\tau)^{2}}e^{-\frac{|x-y|^2}{4(t-\tau)}}\operatorname{d}\!y\operatorname{d}\!\tau\right|  \\
   &\leq& C\varepsilon+ C\varepsilon\int_{|y-x|< \sqrt{\varepsilon}}\frac{O(|x-y|^2)}{|x-y|^2}\operatorname{d}\!y\leq C\varepsilon.
\end{eqnarray*}
Hence, a combination of the estimates of $III_1$ and $III_2$ leads to
\begin{equation}\label{III}
|III|\leq C\sqrt{\varepsilon}.
\end{equation}
Finally, inserting \eqref{I}-\eqref{III} into \eqref{example}, we deduce that
\begin{equation*}
  (\partial_t-\Delta)^{\frac{1}{2}} u(x,t)\geq C_0-C\sqrt{\varepsilon}\geq 0
\end{equation*}
by choosing the positive constant $\varepsilon$ small enough.
\end{proof}

We conclude this section by demonstrating that the nonlocal operator
$(\partial_t-\Delta)^s$ acting on smooth cut-off functions is bounded, which is repeatedly used in establishing our main results.
\begin{lemma}\label{mlem1}
Let
$$\eta(x,t)\in C_0^\infty\left(B_1(0)\times(-1,1)\right)$$
be a smooth cut-off function whose value belongs to $[0,1]$, then there exists a positive constant $C_0$ that depends only on $s$ and $n$ such that
$$\left|(\partial_t-\Delta)^s\eta(x,t)\right|\leq C_0 \,\, \mbox{for}\,\, (x,t)\in B_1(0)\times(-1,1).$$
\end{lemma}
\begin{proof}
For $(x,t)\in B_1(0)\times(-1,1)$, using the definitions of $(\partial_t-\Delta)^s$, we divide the integral domain into three parts
\begin{eqnarray}\label{lem}
 &&(\partial_t-\Delta)^s \eta(x,t)\nonumber\\
 &=&C_{n,s}\int_{-\infty}^{t}\int_{\mathbb{R}^n}
  \frac{\eta(x,t)-\eta(y,\tau)}{(t-\tau)^{\frac{n}{2}+1+s}}e^{-\frac{|x-y|^2}{4(t-\tau)}}\operatorname{d}\!y\operatorname{d}\!\tau \nonumber\\
  &=& C_{n,s}\left(\int_{t-1}^{t}\int_{B_1^c(x)}+\int_{-\infty}^{t-1}\int_{B_1^c(x)}+\int_{-\infty}^{t}\int_{B_1(x)}
  \frac{\eta(x,t)-\eta(y,\tau)}{(t-\tau)^{\frac{n}{2}+1+s}}e^{-\frac{|x-y|^2}{4(t-\tau)}}\operatorname{d}\!y\operatorname{d}\!\tau \right)\nonumber\\
&  =:& I_1+I_2+I_3.
\end{eqnarray}
In order to estimate the first term $I_1$, we combine the smoothness of $\eta(x,t)$ with $0<s<1$ and the fact that
\begin{equation}\label{lem1}
  \frac{e^{-\frac{|z|^2}{4\tau}}}{\tau^{\frac{n}{2}+1+s}}\leq \frac{C}{|z|^{n+2+2s}+\tau^{\frac{n}{2}+1+s}}
\end{equation}
for $\tau>0$ and $z\in \mathbb{R}^n$, where the positive constant $C$ depends on $n$ and $s$, then
\begin{eqnarray*}
  I_1&=& \frac{1}{(4\pi)^{\frac{n}{2}}|\Gamma(-s)|}\int_{t-1}^{t}\int_{B_1^c(x)}
  \frac{\eta(x,t)-\eta(x,\tau)}{(t-\tau)^{\frac{n}{2}+1+s}}e^{-\frac{|x-y|^2}{4(t-\tau)}}\operatorname{d}\!y\operatorname{d}\!\tau  \nonumber\\
   && +C_{n,s}\int_{t-1}^{t}\int_{B_1^c(x)}
  \frac{\eta(x,\tau)-\eta(y,\tau)}{(t-\tau)^{\frac{n}{2}+1+s}}e^{-\frac{|x-y|^2}{4(t-\tau)}}\operatorname{d}\!y\operatorname{d}\!\tau  \nonumber\\
  &\leq&\frac{1}{|\Gamma(-s)|}\int_{t-1}^{t}\frac{|\eta(x,t)-\eta(x,\tau)|}{(t-\tau)^{1+s}}
  \int_{\mathbb{R}^n}
  \frac{e^{-\frac{|x-y|^2}{4(t-\tau)}}}{[4\pi(t-\tau)]^{\frac{n}{2}}}\operatorname{d}\!y\operatorname{d}\!\tau  \nonumber\\
  &&+C\int_{t-1}^{t}\int_{\mathbb{R}^n}
  \frac{|x-y|^2}{|x-y|^{n+2+2s}+(t-\tau)^{\frac{n}{2}+1+s}}\operatorname{d}\!y\operatorname{d}\!\tau  \nonumber\\
   &=&\frac{1}{|\Gamma(-s)|}\int_{t-1}^{t}\frac{|\eta(x,t)-\eta(x,\tau)|}{(t-\tau)^{1+s}}\operatorname{d}\!\tau  +C\int_{0}^{1}\int_{\mathbb{R}^n}
  \frac{|y|^2}{|y|^{n+2+2s}+\tau^{\frac{n}{2}+1+s}}\operatorname{d}\!y\operatorname{d}\!\tau  \nonumber\\
  &\leq&C\int_{t-1}^{t}\frac{(t-\tau)}{(t-\tau)^{1+s}}\operatorname{d}\!\tau  +C\int_{0}^{1}\int_{0}^{+\infty}
  \frac{r^{n+1}}{r^{n+2+2s}+\tau^{\frac{n}{2}+1+s}}\operatorname{d}\!r\operatorname{d}\!\tau\nonumber\\
  &=&\frac{C}{1-s} +C\int_{0}^{1}\int_{0}^{+\infty}
  \frac{r^{n+1}}{r^{n+2+2s}+\tau^{\frac{n}{2}+1+s}}\operatorname{d}\!r\operatorname{d}\!\tau.
\end{eqnarray*}
We further use the formula
\begin{equation}\label{lem2}
  \int_0^{+\infty}\frac{r^q}{a+r^p}\operatorname{d}\!r=\frac{\pi}{p\sin\frac{(q+1)\pi}{p}}a^{\frac{q+1-p}{p}},
\end{equation}
where the positive constants $a$, $p$ and $q$ satisfy $p>q+1$, then
\begin{equation}\label{lem3}
  |I_1| \leq\frac{C}{1-s} +C\int_{0}^{1}
  \tau^{-s}\operatorname{d}\!\tau
   \leq  C(n,s).
\end{equation}
Next, we apply \eqref{lem1} to estimate $I_2$ as follows
\begin{eqnarray}\label{lem4}
  |I_2| &=& \left| C_{n,s}\int_{-\infty}^{t-1}\int_{B_1^c(x)}
  \frac{\eta(x,t)}{(t-\tau)^{\frac{n}{2}+1+s}}e^{-\frac{|x-y|^2}{4(t-\tau)}}\operatorname{d}\!y\operatorname{d}\!\tau \right|\nonumber\\
   &\leq&  C\int_{-\infty}^{t-1}\int_{B_1^c(x)}\frac{1}{|x-y|^{n+2+2s}+(t-\tau)^{\frac{n}{2}+1+s}}
   \operatorname{d}\!y\operatorname{d}\!\tau\leq C(n,s).
\end{eqnarray}
With respect to the estimate of $I_3$, by using the change of variables, Taylor expansion, the definition of Cauchy principal value and \eqref{lem2}, we have
\begin{eqnarray}\label{lem5}
  |I_3| &\leq&  \frac{1}{(4\pi)^{\frac{n}{2}}|\Gamma(-s)|}\left|\int_{-\infty}^{t}\int_{B_1(x)}
  \frac{\eta(x,t)-\eta(y,t)}{(t-\tau)^{\frac{n}{2}+1+s}}e^{-\frac{|x-y|^2}{4(t-\tau)}}\operatorname{d}\!y\operatorname{d}\!\tau \right|\nonumber \\
   &&+  C_{n,s}\left|\int_{-\infty}^{t}\int_{B_1(x)}
  \frac{\eta(y,t)-\eta(y,\tau)}{(t-\tau)^{\frac{n}{2}+1+s}}e^{-\frac{|x-y|^2}{4(t-\tau)}}\operatorname{d}\!y\operatorname{d}\!\tau \right|\nonumber\\
  &\leq&\frac{4^s\Gamma(\frac{n}{2}+s)}{\pi^{\frac{n}{2}}|\Gamma(-s)|}\left|P.V.\int_{B_1(x)}\frac{\eta(x,t)-\eta(y,t)}{|x-y|^{n+2s}}\operatorname{d}\!y\right|\nonumber\\
   &&+C\int_{-\infty}^{t}\int_{B_1(x)}
  \frac{t-\tau}{|x-y|^{n+2+2s}+(t-\tau)^{\frac{n}{2}+1+s}}\operatorname{d}\!y\operatorname{d}\!\tau  \nonumber\\
&\leq&C\int_{B_1(x)}\frac{1}{|x-y|^{n+2s-2}}\operatorname{d}\!y
+C\int_{B_1(0)}\frac{1}{|z|^{n+2s-2}}
  \operatorname{d}\!z \nonumber\\
  &\leq&C(n,s).
\end{eqnarray}
Finally, inserting \eqref{lem3}-\eqref{lem5} into \eqref{lem}, we deduce that
$$\left|(\partial_t-\Delta)^s\eta(x,t)\right|\leq C_0 \,\, \mbox{for}\,\, (x,t)\in B_1(0)\times(-1,1),$$
where the positive constant $C_0=C_0(n,s)$.
Therefore, we complete the proof of Lemma \ref{mlem1}\,.
\end{proof}
As a byproduct of Lemma \ref{mlem1}\,, we can immediately derive the following result through scaling and translation transformations.
\begin{corollary}\label{coro1}
Let $$\eta_r(x,t):=\eta\left(\frac{x-x^0}{r},\frac{t-t_0}{r^{2}}\right)\in C_0^\infty\left(B_r(x^0)\times(-r^{2}+t_0,r^{2}+t_0)\right)$$
for $(x^0,t_0)\in \mathbb{R}^n \times \mathbb{R}$ and $r>0$, then
$$|(\partial_t-\Delta)^s\eta_r(x,t)|\leq \frac{C_0}{r^{2s}} \,\, \mbox{in}\,\, B_r(x^0)\times(-r^{2}+t_0,r^{2}+t_0),$$
where the smooth cut-off function $\eta$ and the positive constant $C_0$ are defined in Lemma \ref{mlem1}\,.
\end{corollary}

\section{Weighted average inequality}\label{2}
In this section, we establish a key estimate for the fully fractional heat operator $(\partial_t-\Delta)^s$, which is commonly referred to as the generalized weighted average inequality (i.e., Theorem \ref{weightAveIneq}). This estimate is particularly useful in establishing the maximum principle and the direct sliding method for master equations, as discussed in later sections.

\begin{proof}[Proof of Theorem \ref{weightAveIneq}]
According to the definition of the nonlocal operator $(\partial_t-\Delta)^s$ and the maximality of $u(x,t)$ at the point $(x^0,t_0)$ in $\mathbb{R}^n\times(-\infty,t_0]$, we derive
\begin{eqnarray*}
  (\partial_t-\Delta)^su(x^0,t_0) &=&C_{n,s}\int_{-\infty}^{t_0}\int_{\mathbb{R}^n}
  \frac{u(x^0,t_0)-u(y,\tau)}{(t_0-\tau)^{\frac{n}{2}+1+s}}e^{-\frac{|x^0-y|^2}{4(t_0-\tau)}}\operatorname{d}\!y\operatorname{d}\!\tau  \\
   &\geq&C_{n,s}\int_{-\infty}^{t_0-r^2}\int_{B_r^c(x^0)}
  \frac{u(x^0,t_0)-u(y,\tau)}{(t_0-\tau)^{\frac{n}{2}+1+s}}e^{-\frac{|x^0-y|^2}{4(t_0-\tau)}}\operatorname{d}\!y\operatorname{d}\!\tau  \\
   &=& \frac{C_{n,s}u(x^0,t_0)}{r^{2s}}\int_{-\infty}^{-1}\int_{B_1^c(0)}
  \frac{e^{\frac{|y|^2}{4\tau}}}{(-\tau)^{\frac{n}{2}+1+s}}\operatorname{d}\!y\operatorname{d}\!\tau   \\
  && - C_{n,s}\int_{-\infty}^{t_0-r^2}\int_{B_r^c(x^0)}
  \frac{u(y,\tau)}{(t_0-\tau)^{\frac{n}{2}+1+s}}e^{-\frac{|x^0-y|^2}{4(t_0-\tau)}}\operatorname{d}\!y\operatorname{d}\!\tau.
\end{eqnarray*}
We denote
$$C_0:=\frac{1}{\int_{-\infty}^{-1}\int_{B^c_{1}(0)}\frac{e^{\frac{|y|^2}{4\tau}}}{(-\tau)^{\frac{n}{2}+1+s}}\operatorname{d}\!y\operatorname{d}\!\tau},$$ then it follows that
\begin{equation*}
  u(x^0,t_0) \leq \frac{C_0}{C_{n,s}}r^{2s}(\partial_t-\Delta)^su(x^0,t_0)+C_0r^{2s}\int_{-\infty}^{t_0-r^2}\int_{B^c_{r}(x^0)}
  \frac{u(y,\tau)e^{-\frac{|x^0-y|^2}{4(t_0-\tau)}}}{(t_0-\tau)^{\frac{n}{2}+1+s}}\operatorname{d}\!y\operatorname{d}\!\tau.
\end{equation*}

It remains to estimate validity of \eqref{measure}. In terms of the definition of $C_0$, a direct calculation shows that
\begin{eqnarray*}
   && C_0r^{2s}\int_{-\infty}^{t_0-r^2}\int_{B^c_{r}(x^0)}
  \frac{e^{-\frac{|x^0-y|^2}{4(t_0-\tau)}}}{(t_0-\tau)^{\frac{n}{2}+1+s}}\operatorname{d}\!y\operatorname{d}\!\tau \\
   &=& C_0\int_{-\infty}^{-1}\int_{B^c_{1}(0)}
  \frac{e^{\frac{|y|^2}{4\tau}}}{(-\tau)^{\frac{n}{2}+1+s}}\operatorname{d}\!y\operatorname{d}\!\tau  =1.
\end{eqnarray*}
Thus, we complete the proof of Theorem \ref{weightAveIneq}\,.
\end{proof}

\section{Maximum principle in unbounded domains and its direct application}\label{3}

In this section, we demonstrate the maximum principle in unbounded domains (i.e., Theorem \ref{MPUB}) for master equations by utilizing the perturbation argument as well as the weighted average inequality stated in Theorem \ref{weightAveIneq}\,. Furthermore, we propose that a direct application of this maximum principle is to establish the monotonicity of solutions to master equations on an upper half space. More importantly, it serves as a fundamental ingredient in carrying out the sliding method adopted in the proof of Gibbons' conjecture.

\begin{proof}[Proof of Theorem \ref{MPUB}] We argue by contradiction, if \eqref{MPUB3} is violated, since $u(x,t)$ has an upper bound in $\Omega\times\mathbb{R}$, then there exists a positive constant $A$ such that
\begin{equation}\label{MPUB4}
  \sup_{\Omega\times\mathbb{R}}u(x,t):=A>0.
\end{equation}
Note that the set $\Omega\times\mathbb{R}$ is  unbounded, then the supremum of $u(x,t)$ may not be attained.
Even so, \eqref{MPUB4} implies that there exists a sequence $\{(x^k,t_k)\}\subset \Omega\times\mathbb{R}$ such that
$$0<u(x^k,t_k):=A_k\rightarrow A,\,\, \mbox{as}\,\, k\rightarrow\infty.$$
Let $\varepsilon_k:=A-A_k$, then the sequence $\{\varepsilon_k\}$ is nonnegative and
tends to zero as $k\rightarrow\infty$.
To proceed, we introduce the following auxiliary function
$$v_k(x,t):=u(x,t)+\varepsilon_k\eta_k(x,t),$$
where the smooth cut-off function
$$\eta_k(x,t):=\eta\left(\frac{x-x^k}{r},\frac{t-t_k}{r^2}\right)\in C_0^\infty(B_r(x^k)\times(t_k-r^2,t_k+r^2)),$$
 satisfying
\begin{equation*}
\eta(x,t):=\left\{
\begin{array}{ll}
1  & (x,t)\in B_{\frac{1}{2}}(0)\times(-\frac{1}{2},\frac{1}{2}), \\
  0 , ~ &(x,t)\not\in B_{1}(0)\times(-1,1).
\end{array}
\right.
\end{equation*}
We first determine the radius $r$ in the scaled and translated smooth function $\eta_k$.
In terms of the condition \eqref{MPUB1}, we can directly evaluate
\begin{equation*}
  \lim_{R\rightarrow +\infty}\frac{\left|(B_{3R}(x^k)\setminus B_{\frac{3R}{\sqrt[n]{2}}}(x^k))\cap \Omega^c\right|}{|B_{3R}(x^k)|}\geq\frac{\bar{C}}{2}>0.
\end{equation*}
It follows that there exists a sufficiently large radius $R_k$ such that
\begin{equation}\label{MPUB8}
  \frac{\left|(B_{3R}(x^k)\setminus B_{\frac{3R}{\sqrt[n]{2}}}(x^k))\cap \Omega^c\right|}{|B_{3R}(x^k)|}\geq\frac{\bar{C}}{4}>0 \,\, \mbox{for}\,\, R\geq R_k.
\end{equation}
From then on, we select the radius
$$r=\frac{R_k}{\sqrt[n]{2}}.$$
Let
$$Q_r(x^k,t_k):=B_{r}(x^k)\times (t_k-r^{2},t_k+r^{2})$$
be a parabolic cylinder, then a straightforward calculation implies that
$$v_k(x^k,t_k)=u(x^k,t_k)+\varepsilon_k=A_k+A-A_k=A,$$
and
$$v_k(x,t)=u(x,t)\leq A\,\, \mbox{for}\,\, (x,t)\not\in Q_r(x^k,t_k).$$
It is evident that the auxiliary function $v_k(x,t)$ must attain its maximum value in $Q_r(x^k,t_k)$. More precisely, there exists a point $(\bar{x}^k,\bar{t}_k)\in Q_r(x^k,t_k)$ such that
\begin{equation}\label{MPUB5}
 A+\varepsilon_k \geq v_k(\bar{x}^k,\bar{t}_k)=\sup_{\mathbb{R}^n\times\mathbb{R}}v_k(x,t)\geq A>0.
\end{equation}
Furthermore, by virtue of the definition of $v_k$, we derive
$$A\geq u(\bar{x}^k,\bar{t}_k)\geq A-\varepsilon_k=A_k>0.$$

Combining the definition of $v_k$ with the differential equation in \eqref{MPUB2} and Corollary \ref{coro1}\,, we obtain
\begin{equation}\label{MPUB6}
 (\partial_t-\Delta)^sv_k(\bar{x}^k,\bar{t}_k)=(\partial_t-\Delta)^su(\bar{x}^k,\bar{t}_k)
 +\varepsilon_k(\partial_t-\Delta)^s\eta_k(\bar{x}^k,\bar{t}_k)\leq\frac{C\varepsilon_k}{r^{2s}}.
\end{equation}
Now applying the weighted average inequality established in Theorem \ref{weightAveIneq} to $v_k$ at its maximum point $(\bar{x}^k,\bar{t}_k)$ and combining with \eqref{MPUB5} and \eqref{MPUB6}, we have
\begin{eqnarray}\label{MPUB7}
  A &\leq &v_k(\bar{x}^k,\bar{t}_k)\nonumber \\
&\leq & \frac{C_0}{C_{n,s}}(2r)^{2s}(\partial_t-\Delta)^sv_k(\bar{x}^k,\bar{t}_k) +C_0(2r)^{2s}\int_{-\infty}^{\bar{t}_k-(2r)^2}\int_{B^c_{2r}(\bar{x}^k)}
  \frac{v_k(y,\tau)e^{-\frac{|\bar{x}^k-y|^2}{4(\bar{t}_k-\tau)}}}{(\bar{t}_k-\tau)^{\frac{n}{2}+1+s}}\operatorname{d}\!y\operatorname{d}\!\tau\nonumber\\
  &\leq& C\varepsilon_k+C_0(2r)^{2s}\int_{-\infty}^{\bar{t}_k-(2r)^2}\int_{B^c_{2r}(\bar{x}^k)}
  \frac{v_k(y,\tau)e^{-\frac{|\bar{x}^k-y|^2}{4(\bar{t}_k-\tau)}}}{(\bar{t}_k-\tau)^{\frac{n}{2}+1+s}}\operatorname{d}\!y\operatorname{d}\!\tau.
\end{eqnarray}
It remains to be estimated the second term on the left side of \eqref{MPUB7}.
A combination of the containment relationship of balls
\begin{equation}\label{MPUB7-1}
B_r(x^k)\subset B_{2r}(\bar{x}^k)\subset B_{3r}(x^k)
\end{equation}
with the exterior condition of $u$ in \eqref{MPUB2} and the definition of the smooth function $\eta_k$ yields that
\begin{equation}\label{MPUB7-2}
  v_k(x,t)=u(x,t)+\varepsilon_k\eta_k(x,t)=u(x,t)\leq 0,\,\, \mbox{in}\,\, (B_{2r}^c(\bar{x}^k)\cap\Omega^c)\times \mathbb{R}.
\end{equation}
Next, applying \eqref{MPUB8}, \eqref{MPUB5}, \eqref{MPUB7-1} and \eqref{MPUB7-2}, and combining with the definition of $C_0$ presented in Theorem \ref{weightAveIneq} and the chosen of the radius $r=\frac{R_k}{\sqrt[n]{2}}$,
we estimate the second term on the left side of \eqref{MPUB7} as follows
\begin{eqnarray}\label{MPUB9}
   && C_0(2r)^{2s}\int_{-\infty}^{\bar{t}_k-(2r)^2}\int_{B^c_{2r}(\bar{x}^k)}
  \frac{v_k(y,\tau)e^{-\frac{|\bar{x}^k-y|^2}{4(\bar{t}_k-\tau)}}}{(\bar{t}_k-\tau)^{\frac{n}{2}+1+s}}\operatorname{d}\!y\operatorname{d}\!\tau \nonumber\\
  &=&  C_0(2r)^{2s}\int_{-\infty}^{\bar{t}_k-(2r)^2}\left[\int_{B^c_{2r}(\bar{x}^k)\cap\Omega}
  \frac{v_k(y,\tau)e^{-\frac{|\bar{x}^k-y|^2}{4(\bar{t}_k-\tau)}}}{(\bar{t}_k-\tau)^{\frac{n}{2}+1+s}}\operatorname{d}\!y+ \int_{B^c_{2r}(\bar{x}^k)\cap\Omega^c}
  \frac{(A+\varepsilon_k)e^{-\frac{|\bar{x}^k-y|^2}{4(\bar{t}_k-\tau)}}}{(\bar{t}_k-\tau)^{\frac{n}{2}+1+s}}\operatorname{d}\!y \right. \nonumber\\
   &&\left.-\int_{B^c_{2r}(\bar{x}^k)\cap\Omega^c}
  \frac{(A+\varepsilon_k)e^{-\frac{|\bar{x}^k-y|^2}{4(\bar{t}_k-\tau)}}}{(\bar{t}_k-\tau)^{\frac{n}{2}+1+s}}\operatorname{d}\!y +\int_{B^c_{2r}(\bar{x}^k)\cap\Omega^c}
  \frac{v_k(y,\tau)e^{-\frac{|\bar{x}^k-y|^2}{4(\bar{t}_k-\tau)}}}{(\bar{t}_k-\tau)^{\frac{n}{2}+1+s}}\operatorname{d}\!y\right]\operatorname{d}\!\tau\nonumber\\
  &\leq&C_0(2r)^{2s}\int_{-\infty}^{\bar{t}_k-(2r)^2}\int_{B^c_{2r}(\bar{x}^k)}
  \frac{(A+\varepsilon_k)e^{-\frac{|\bar{x}^k-y|^2}{4(\bar{t}_k-\tau)}}}{(\bar{t}_k-\tau)^{\frac{n}{2}+1+s}}\operatorname{d}\!y\operatorname{d}\!\tau\nonumber\\
  &&-C_0(2r)^{2s}\int_{-\infty}^{\bar{t}_k-(3r)^2}\int_{B^c_{3r}(x^k)\cap\Omega^c}
  \frac{(A+\varepsilon_k)e^{-\frac{|\bar{x}^k-y|^2}{4(\bar{t}_k-\tau)}}}{(\bar{t}_k-\tau)^{\frac{n}{2}+1+s}}\operatorname{d}\!y\operatorname{d}\!\tau\nonumber\\
&=&A+\varepsilon_k-C_0(2r)^{2s}\int_{-\infty}^{\bar{t}_k-(3r)^2}\int_{B^c_{3r}(x^k)\cap\Omega^c}
  \frac{(A+\varepsilon_k)e^{-\frac{|\bar{x}^k-y|^2}{4(\bar{t}_k-\tau)}}}{(\bar{t}_k-\tau)^{\frac{n}{2}+1+s}}\operatorname{d}\!y\operatorname{d}\!\tau\nonumber\\
  &\leq&A+\varepsilon_k-C_0(\frac{2R_k}{\sqrt[n]{2}})^{2s}\int_{\bar{t}_k-(3R_k)^2}^{\bar{t}_k-(\frac{3R_k}{\sqrt[n]{2}})^2}\int_{(B_{3R_k}(x^k)\setminus B^c_{\frac{3R_k}{\sqrt[n]{2}}}(x^k))\cap\Omega^c}
  \frac{(A+\varepsilon_k)e^{-\frac{|\bar{x}^k-y|^2}{4(\bar{t}_k-\tau)}}}{(\bar{t}_k-\tau)^{\frac{n}{2}+1+s}}\operatorname{d}\!y\operatorname{d}\!\tau\nonumber\\
&\leq&A+\varepsilon_k-C_0(\frac{2R_k}{\sqrt[n]{2}})^{2s} (A+\varepsilon_k)e^{-\frac{(3\sqrt[n]{2}+1)^2}{36}}\frac{(3R_k)^2-(\frac{3R_k}{\sqrt[n]{2}})^2}{(3R_k)^{n+2+2s}}\left|(B_{3R_k}(x^k)\setminus B^c_{\frac{3R_k}{\sqrt[n]{2}}}(x^k))\cap\Omega^c\right|
\nonumber\\
&\leq&A+\varepsilon_k-C_0(\frac{2R_k}{\sqrt[n]{2}})^{2s} (A+\varepsilon_k)e^{-\frac{(3\sqrt[n]{2}+1)^2}{36}}\frac{(3R_k)^2-(\frac{3R_k}{\sqrt[n]{2}})^2}{(3R_k)^{n+2+2s}}\frac{\bar{C}}{4}\left|B_{3R_k}(x^k)\right|
\nonumber\\
&\leq&(1-C)(A+\varepsilon_k).
\end{eqnarray}

Finally, substituting \eqref{MPUB9} into \eqref{MPUB7}, we deduce that
$$0<A\leq C\varepsilon_k+ (1-C) A,$$
which is a contradiction for sufficiently large $k$, and thus completes the proof of Theorem \ref{MPUB}\,.
\end{proof}
It is well known that the maximum principle is a powerful tool and has many
straightforward applications in the study of partial differential equations. For instance,
Theorem \ref{MPUB} can be directly used to establish Corollary \ref{Application}\,, that is
the monotonicity of solutions to master equations on an upper half space.
\begin{proof}[Proof of Corollary \ref{Application}]
From now on, for any $\lambda>0$, let
$$u_{\lambda}(x,t):=u(x^\lambda,t),$$
where $x^\lambda:=x+\lambda e_n$ with $e_n=(0,...,0,1)$, and
$$w_{\lambda}(x,t):=u(x,t)-u_{\lambda}(x,t),$$
where the positions of the points $x$ and $x^\lambda$ as illustrated in Figure 4 below.
\begin{center}
\begin{tikzpicture}[scale=0.8]
\draw [very thick]  [black] [->,very thick](-4,0)--(4,0) node [anchor=north west] {$x'$};
\draw [very thick]  [black!80][->,very thick] (0,-1)--(0,4) node [black][ above] {$x_n$};
\path node at (-0.3,-0.3) {$0$};
\path (1,1)[very thick,fill=red]  circle(1.8pt) node at (1.3,1) {$x$};
\path (1,3)[very thick,fill=red]  circle(1.8pt) node at (1.4,3) {$x^\lambda$};
\draw [very thick] [dashed] [blue] (1,1)--(1,3);
\node [below=0.5cm, align=flush center,text width=12cm] at  (0,-0.5)
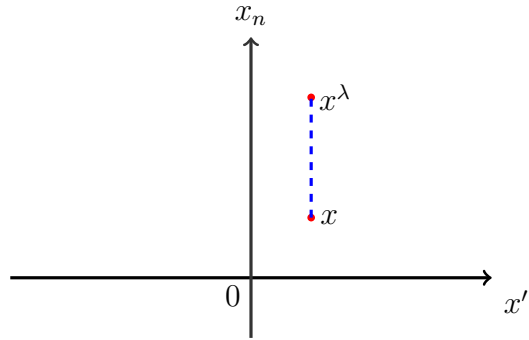
        {Figure 4. The positions of the points. };
\end{tikzpicture}
\end{center}
If $u(x,t)$ is not increasing with respect to $x_n$ in $\mathbb{R}^n_+$ for any $t\in\mathbb{R}$, then there exists some point $(x,t)\in\mathbb{R}^n_+\times\mathbb{R}$ such that $w_\lambda(x,t)>0$.
Since $f(t,u)$ is monotonically decreasing with respect to $u$, we derive
\begin{eqnarray}\label{AP2}
 (\partial_t-\Delta)^sw_\lambda(x,t)&=&(\partial_t-\Delta)^su(x,t)-(\partial_t-\Delta)^su_\lambda(x,t)\nonumber\\
 &=&f(t,u(x,t))-f(t,u_\lambda(x,t))\leq 0
\end{eqnarray}
in $\mathbb{R}^n_+\times\mathbb{R}$ at the points where $w_\lambda(x,t)>0$.
Meanwhile, a combination the exterior condition in \eqref{AP1} and the nonnegativity of $u(x,t)$ yields that
\begin{equation*}
  w_\lambda(x,t)=u(x,t)-u_{\lambda}(x,t)=-u_{\lambda}(x,t)\leq 0 \,\, \mbox{in}\,\, (\mathbb{R}^n\setminus\mathbb{R}^n_+)\times\mathbb{R}.
\end{equation*}
Then by virtue of Theorem \ref{MPUB} with $\Omega=\mathbb{R}^n_+$, we conclude that
\begin{equation*}
  w_\lambda(x,t)\leq 0 \,\, \mbox{in}\,\, \mathbb{R}^n_+\times\mathbb{R}
\end{equation*}
for any $\lambda>0$. It implies that $u(x,t)$ must be increasing with respect to $x_n$ in $\mathbb{R}^n_+$ for any $t\in\mathbb{R}$.

In the sequel, we further prove that the increase of $u(x,t)$ is strict, it suffices to claim that
\begin{equation}\label{AP3}
  w_\lambda(x,t)< 0 \,\, \mbox{in}\,\, \mathbb{R}^n_+\times\mathbb{R}
\end{equation}
for any $\lambda>0$. If not, then there exist some $\lambda_0>0$ and a point $(x^0,t_0)\in \mathbb{R}^n_+\times\mathbb{R}$ such that
\begin{equation*}
  w_{\lambda_0}(x^0,t_0)=0=\max_{\mathbb{R}^n\times\mathbb{R}}  w_{\lambda_0}(x,t).
\end{equation*}
Combining the differential inequality \eqref{AP2} with the definition of the nonlocal operator $(\partial_t-\Delta)^s$, we deduce
\begin{equation*}
  (\partial_t-\Delta)^s w_{\lambda_0}(x^0,t_0)= C_{n,s}\int_{-\infty}^{t_0}\int_{\mathbb{R}^n}
  \frac{- w_{\lambda_0}(y,\tau)}{(t_0-\tau)^{\frac{n}{2}+1+s}}e^{-\frac{|x^0-y|^2}{4(t_0-\tau)}}\operatorname{d}\!y\operatorname{d}\!\tau=0.
\end{equation*}
By using $w_{\lambda_0}(x,t)\leq 0$ in $\mathbb{R}^n\times\mathbb{R}$, then we must have
$w_{\lambda_0}(x,t)\equiv 0$ in $\mathbb{R}^n\times(-\infty,t_0]$, which contradicts the fact that $w_{\lambda_0}(x,t)\not\equiv 0$ in $\mathbb{R}^n$ for any fixed $t\in (-\infty,t_0]$ due to the exterior condition and the interior positivity of $u(x,t)$. Hence, we verify that the assertion \eqref{AP3} is valid, then the proof of Corollary \ref{Application} is complete.
\end{proof}

\section{The proof of Gibbons' conjecture for master equations}\label{4}

This section provides a complete proof of Gibbons' conjecture for master equations, which is based on a direct sliding method applicable to the master equation version, with the weighted average inequality (i.e., Theorem \ref{weightAveIneq}) and the maximum principle in unbounded domains (i.e., Theorem \ref{MPUB}) as effective techniques throughout.

For the ease of readability, we now present an outline of the proof.
Keeping the notations $x^\lambda$, $u_{\lambda}$ and $w_{\lambda}$ defined in the proof of Corollary \ref{Application}\,, we proceed in three steps and first argue that the assertion
\begin{equation}\label{Main1-1}
w_\lambda(x,t)\leq0 \,\,\mbox{in}\,\, \mathbb{R}^n\times\mathbb{R}
\end{equation}
is valid for sufficiently large $\lambda$ by using the maximum principle in unbounded domains.
The first Step provides a starting point for sliding the domain upwards along the $x_n$-axis.

In the second step, we continuously decrease $\lambda$ to its limiting position as long as the inequality \eqref{Main1-1} holds. We aim to
show that the domain can slide back to overlap with the original domain.
Otherwise, we will utilize the maximum principle in unbounded domains again and combine the weighted average inequality with the perturbation technique and the limit argument to derive a contradiction. Based on \eqref{Main1-1} for any $\lambda>0$, we further deduce that $u(x,t)$ is strictly increasing with respect to $x_{n}$ for any $t\in\mathbb{R}$ by establishing a strong maximum principle
\begin{equation*}
  w_\lambda(x,t)< 0, \,\,\mbox{in}\,\, \mathbb{R}^n\times\mathbb{R} \,\, \mbox{for any}\,\, \lambda>0.
\end{equation*}

In the final step, we apply the sliding method along any direction that has an acute angle with the positive $x_n$ axis to demonstrate that the entire solution $u(x,t)$ is one-dimensional symmetry for any $t\in\mathbb{R}$, that is,
$$u(x',x_n,t)=u(x_n,t)\,\, \mbox{for any}\,\, t\in\mathbb{R}.$$

Next, we provide the detailed process of proving Gibbons' conjecture for master equations.

\begin{proof}[Proof of Theorem \ref{Mainresult}]
The proof is divided into three steps.

\noindent \textbf{Step 1.} We first show that
\begin{equation}\label{Main1}
w_\lambda(x,t)\leq0 \,\,\mbox{in}\,\, \mathbb{R}^n\times\mathbb{R}
\end{equation}
for sufficiently large $\lambda$.

The uniform convergence condition of $u$ in Theorem \ref{Mainresult} implies that there exists a sufficiently large $a>0$ such that
\begin{equation}\label{Main2}
  |u(x,t)|\geq 1-\delta \,\,\mbox{for}\,\, |x_n|\geq a\,\, \mbox{and}\,\, (x',t)\in\mathbb{R}^{n-1}\times\mathbb{R}.
\end{equation}
In order to prove \eqref{Main1}, assume on the contrary that there exists a positive constant $A$ such that
\begin{equation}\label{Main3}
\sup_{\mathbb{R}^n\times\mathbb{R}}w_\lambda(x,t)=A>0
\end{equation}
for any sufficiently large $\lambda$.
Let the auxiliary function
$$\bar{w}_\lambda(x,t):=w_\lambda(x,t)-\frac{A}{2},$$
now we claim that
\begin{equation}\label{Main4}
\bar{w}_\lambda(x,t)\leq0 \,\,\mbox{in}\,\, \mathbb{R}^n\times\mathbb{R}
\end{equation}
for sufficiently large $\tau$.
We further choose a sufficiently large constant $M>a$ such that $$\bar{w}_\lambda(x,t)\leq 0 \,\,\mbox{for any} \,\, (x,t)\in \mathbb{R}^{n}\times\mathbb{R}\,\, \mbox{with}\,\, x_n\geq M.$$
We denote the unbounded domain $\Omega:=\mathbb{R}^{n-1}\times(-\infty,M)$, then it follows that
$$\bar{w}_\lambda(x,t)\leq 0 \,\, \mbox{in} \,\, \Omega^c\times\mathbb{R},$$
which is in accordance with the exterior condition in Theorem \ref{MPUB}\,.

Next, we devote to verifying
\begin{equation}\label{Main5}
  (\partial_t-\Delta)^s\bar{w}_\lambda(x,t)=(\partial_t-\Delta)^s w_\lambda(x,t)=f(t,u(x,t))-f(t,u_\lambda(x,t))\leq0
\end{equation}
at the points in $\Omega\times\mathbb{R}$ where $\bar{w}_{\lambda}(x,t)>0$ for sufficiently large $\lambda$, which satisfies the differential inequality in Theorem \ref{MPUB}\,.

We distinguish three cases and first argue that the assertion \eqref{Main5} is true for $|x_n|\leq a$ with any $\lambda\geq2a$. In this case, we must have $x_n+\lambda\geq a$, then \eqref{Main2} indicates that
$$1\geq u(x,t)>u_{\lambda}(x,t)+\frac{A}{2}>u_{\lambda}(x,t)\geq1-\delta$$
for any $(x',t)\in\mathbb{R}^{n-1}\times\mathbb{R}$ at the points where $\bar{w}_{\lambda}(x,t)>0$. Thereby for any fixed $t\in\mathbb{R}$, applying the non-increasing assumption on $f(t,u)$ for $u\in[1-\delta,1]$, we derive
\begin{equation*}
  (\partial_t-\Delta)^s\bar{w}_\lambda(x,t)=f(t,u(x,t))-f(t,u_\lambda(x,t))\leq0
\end{equation*}
at the points in $\Omega\times\mathbb{R}$ with $|x_n|\leq a$ where $\bar{w}_{\lambda}(x,t)>0$ for any $\lambda\geq2a$.

While if $x_n<-a$, then by virtue of \eqref{Main2}, we obtain
\begin{equation*}
 -1\leq u_\lambda(x,t)<u(x,t)-\frac{A}{2}<u(x,t)\leq-1+\delta
\end{equation*}
for any $(x',t)\in\mathbb{R}^{n-1}\times\mathbb{R}$ at the points where $\bar{w}_{\lambda}(x,t)>0$. Using the non-increasing assumption on $f(t,u)$ with respect to $u\in[-1,-1+\delta]$ for any fixed $t\in\mathbb{R}$, we derive
\begin{equation*}
  (\partial_t-\Delta)^s\bar{w}_\lambda(x,t)=f(t,u(x,t))-f(t,u_\lambda(x,t))\leq0
\end{equation*}
at the points in $\Omega\times\mathbb{R}$ with $x_n<-a$ where $\bar{w}_{\lambda}(x,t)>0$.

The last case is $x_n\in(a,M)$, by virtue of \eqref{Main2} again, we have
$$1\geq u(x,t)>u_{\lambda}(x,t)+\frac{A}{2}>u_{\lambda}(x,t)\geq1-\delta$$
for any $(x',t)\in\mathbb{R}^{n-1}\times\mathbb{R}$ at the points where $\bar{w}_{\lambda}(x,t)>0$. Then for any fixed $t\in\mathbb{R}$, applying the non-increasing assumption on $f(t,u)$ for $u\in[1-\delta,1]$ again, we deduce that the assertion \eqref{Main5} is valid for $x_n\in(a,M)$.

In conclusion, we show that $\bar{w}_{\lambda}(x,t)$ satisfies
\begin{equation*}
\left\{
\begin{array}{ll}
    (\partial_t-\Delta)^s\bar{w}_{\lambda}(x,t)\leq 0 ,~   & \mbox{at the points in}\,\, \Omega\times\mathbb{R} \,\,\mbox{where}\,\, \bar{w}_{\lambda}(x,t)>0 , \\
  \bar{w}_{\lambda}(x,t)\leq 0 , ~ &\mbox{in}\,\, \Omega^c \times\mathbb{R},
\end{array}
\right.
\end{equation*}
for any $\lambda\geq2a$. Here  the unbounded domain $\Omega=\mathbb{R}^{n-1}\times(-\infty,M)$ clearly satisfies the limit condition \eqref{MPUB1} stated in Theorem \ref{MPUB}\,.
Thus, Theorem \ref{MPUB} infers that $\bar{w}_\lambda(x,t)\leq 0$ in $\mathbb{R}^n\times\mathbb{R}$ for any $\lambda\geq2a$. That is to say,
$$w_\lambda(x,t)\leq \frac{A}{2}\,\, \mbox{in} \,\, \mathbb{R}^n\times\mathbb{R}$$
for any $\lambda\geq2a$, which is a contradiction with \eqref{Main3}. Hence, we verify that
\begin{equation*}
w_\lambda(x,t)\leq0, \, (x,t) \in \mathbb{R}^n\times\mathbb{R}
\end{equation*}
for any $\lambda\geq2a$. Indeed, Step 1 provides a starting point for sliding the domain along the $x_n$-axis.

\noindent \textbf{Step 2.} In this step, we continuously decrease $\lambda$ to its limiting position as long as the inequality \eqref{Main1} holds and define
\begin{equation*}
  \lambda_0:=\inf\{\lambda\mid w_\lambda(x,t)\leq0, \, (x,t)\in \mathbb{R}^n\times\mathbb{R} \}.
\end{equation*}
We devote to proving that the limiting position
$$\lambda_0=0.$$
The argument goes by contradiction, if not, then $\lambda_0>0$,
we show that $\lambda_0$ can be decreased a little bit
while the inequality \eqref{Main1} is still valid in this case, which contradicts the definition of $\lambda_0$.

With this aim in mind, we first claim that
\begin{equation}\label{Main6}
  \sup_{(x,t)\in \left(\mathbb{R}^{n-1}\times[-a,a]\right)\times \mathbb{R}} w_{\lambda_0}(x,t)<0,
\end{equation}
where the sufficiently large $a>0$ is defined in \eqref{Main2}.
If the assertion \eqref{Main6} is violated, then we have
\begin{equation*}
  \sup_{(x,t)\in \left(\mathbb{R}^{n-1}\times[-a,a]\right)\times \mathbb{R}} w_{\lambda_0}(x,t)=0,
\end{equation*}
which means that there exist sequences $\{(x^k,t_k)\}\subset \left(\mathbb{R}^{n-1}\times[-a,a]\right)\times \mathbb{R}$ and $\{\varepsilon_k\}\subset\mathbb{R}_+$ such that
\begin{equation*}
 w_{\lambda_0}(x^k,t_k)=:-\varepsilon_k\rightarrow 0,\,\, \mbox{as}\,\, k\rightarrow \infty.
\end{equation*}
We further introduce the following auxiliary function
\begin{equation*}
  w_k(x,t):= w_{\lambda_0}(x,t)+\varepsilon_k\eta_k(x,t)
\end{equation*}
to remedy the supremum of $w_{\lambda_0}(x,t)$ may not be attained due to the set $\left(\mathbb{R}^{n-1}\times[-a,a]\right)\times \mathbb{R}$ is unbounded.
Here
$$\eta_k(x,t)=\eta(x-x^k,t-t_k)\in C_0^\infty\left(B_1(x^k)\times(-1+t_k,1+t_k)\right)$$
is a smooth cut-off function satisfying
$$\eta_k(x,t)\equiv1\,\, \mbox{in}\,\,B_{\frac{1}{2}}(x^k)\times\left(-\frac{1}{2}+t_k,\frac{1}{2}+t_k\right),\,\,\mbox{and}\,\, 0\leq\eta_k(x,t)\leq 1.$$
Let the parabolic cylinder
$$Q_{1}(x^k,t_k):=B_{1}(x^k)\times (-1+t_k,1+t_k),$$
then through a direct calculation, we obtain
\begin{equation*}
  w_k(x^k,t_k)=w_{\lambda_0}(x^k,t_k)+\varepsilon_k\eta_k(x^k,t_k)
  =-\varepsilon_k+\varepsilon_k=0,
\end{equation*}
 and
\begin{equation*}
  w_k(x,t)=w_{\lambda_0}(x,t)\leq 0, \,\, \mbox{in}\,\, (\mathbb{R}^n\times \mathbb{R}) \setminus Q_{1}(x^k,t_k).
\end{equation*}
Hence, the perturbed function $w_k(x,t)$ can attain its maximum value at some point $(\bar{x}^k,\bar{t}_k)\in Q_{1}(x^k,t_k)$ such that
\begin{equation}\label{Main7}
\varepsilon_k  \geq w_k(\bar{x}^k,\bar{t}_k)=\sup_{\mathbb{R}^n\times \mathbb{R}}w_k(x,t)\geq0.
\end{equation}
Moreover, the definition of $w_k(x,t)$ yields that
\begin{equation}\label{Main8}
  0\geq w_{\lambda_0}(\bar{x}^k,\bar{t}_k)\geq-\varepsilon_k.
\end{equation}

For the sake of illustration, we introduce the translation function
\begin{equation*}
  \bar{w}_k(x,t):=w_k(x+\bar{x}^k,t+\bar{t}_k).
\end{equation*}
It follows from \eqref{Main7} that
\begin{equation}\label{Main9}
\varepsilon_k  \geq \bar{w}_k(0,0)=\sup_{\mathbb{R}^n\times \mathbb{R}}\bar{w}_k(x,t)\geq0,
\end{equation}
and then
\begin{equation*}
  (\partial-\Delta)^s\bar{w}_k(0,0)=C_{n,s}\int_{-\infty}^{0}\int_{\mathbb{R}^n}
  \frac{\bar{w}_k(0,0)- \bar{w}_k(y,\tau)}{(-\tau)^{\frac{n}{2}+1+s}}e^{\frac{|y|^2}{4\tau}}\operatorname{d}\!y\operatorname{d}\!\tau\geq 0.
\end{equation*}
Next, we aim to claim that
\begin{equation}\label{Main10}
  (\partial-\Delta)^s\bar{w}_k(0,0)\rightarrow 0\,\, \mbox{as}\,\, k\rightarrow\infty.
\end{equation}
Combining the equation satisfied by $u$ with Lemma \ref{mlem1}\,, we compute
\begin{eqnarray*}
 0\leq (\partial-\Delta)^s\bar{w}_k(0,0) &=& (\partial-\Delta)^sw_k(\bar{x}^k,\bar{t}_k)\nonumber \\
   &=&(\partial-\Delta)^sw_{\lambda_0}(\bar{x}^k,\bar{t}_k)+\varepsilon_k(\partial-\Delta)^s\eta_k(\bar{x}^k,\bar{t}_k)  \nonumber\\
   &\leq& f(\bar{t}_k,u(\bar{x}^k,\bar{t}_k))-f(\bar{t}_k,u_{\lambda_0}(\bar{x}^k,\bar{t}_k))+C\varepsilon_k\nonumber \\
  &\rightarrow&  0\,\, \mbox{as}\,\, k\rightarrow\infty,
\end{eqnarray*}
where the last line we use the continuity of $f$ and the fact that $$w_{\lambda_0}(\bar{x}^k,\bar{t}_k)=u(\bar{x}^k,\bar{t}_k)-u_{\lambda_0}(\bar{x}^k,\bar{t}_k)\rightarrow 0 \,\, k\rightarrow\infty$$
by \eqref{Main8}. Thus, we verify that the assertion \eqref{Main10} is valid.

Now applying Theorem \ref{weightAveIneq} to $\bar{w}_k(x,t)$ at its maximum point $(0,0)$, we have
\begin{equation}\label{Main11}
  \bar{w}_k(0,0) \leq \frac{ C_0}{C_{n,s}}r^{2s}(\partial_t-\Delta)^s\bar{w}_k(0,0)+C_0r^{2s}\int_{-\infty}^{-r^2}\int_{B^c_{r}(0)}
  \frac{\bar{w}_k(y,\tau)e^{\frac{|y|^2}{4\tau}}}{(-\tau)^{\frac{n}{2}+1+s}}\operatorname{d}\!y\operatorname{d}\!\tau
\end{equation}
for any $r>0$, where the positive constant $C_0$ is defined in Theorem \ref{weightAveIneq}\,.
If we select the radius $r>2$ such that the point $(y+\bar{x}^k,\tau+\bar{t}_k)\not\in Q_{1}(x^k,t_k)$ for $(y,\tau)\in B_r^c(0)\times(-\infty,-r^2)$, then
\begin{eqnarray*}
   && C_0r^{2s}\int_{-\infty}^{-r^2}\int_{B^c_{r}(0)}
  \frac{\bar{w}_k(y,\tau)}{(-\tau)^{\frac{n}{2}+1+s}}e^{\frac{|y|^2}{4\tau}}\operatorname{d}\!y\operatorname{d}\!\tau \\
   &=& C_0r^{2s}\int_{-\infty}^{-r^2}\int_{B^c_{r}(0)}
  \frac{w_{\lambda_0}(y+\bar{x}^k,\tau+\bar{t}_k)+\varepsilon_k \eta_{k}(y+\bar{x}^k,\tau+\bar{t}_k)  }{(-\tau)^{\frac{n}{2}+1+s}}e^{\frac{|y|^2}{4\tau}}\operatorname{d}\!y\operatorname{d}\!\tau\\
&=& C_0r^{2s}\int_{-\infty}^{-r^2}\int_{B^c_{r}(0)}
  \frac{w_{\lambda_0}(y+\bar{x}^k,\tau+\bar{t}_k)  }{(-\tau)^{\frac{n}{2}+1+s}}e^{\frac{|y|^2}{4\tau}}\operatorname{d}\!y\operatorname{d}\!\tau\leq 0
\end{eqnarray*}
by the definition of $\lambda_0$. Thereby a combination of \eqref{Main9}-\eqref{Main11} leads to
\begin{equation*}
  C_0r^{2s}\int_{-\infty}^{-r^2}\int_{B^c_{r}(0)}
  \frac{\bar{w}_k(y,\tau)}{(-\tau)^{\frac{n}{2}+1+s}}e^{\frac{|y|^2}{4\tau}}\operatorname{d}\!y\operatorname{d}\!\tau \rightarrow 0 \,\,\mbox{as}\,\, k\rightarrow\infty,
\end{equation*}
which indicates that
\begin{equation}\label{Main12}
  \bar{w}_k(x,t)\rightarrow 0 \,\, \mbox{for}\,\, (x,t)\in B_r^c(0)\times(-\infty,-r^2), \,\,\mbox{as}\,\, k\rightarrow\infty.
\end{equation}

We further take the same translation for $u$ as follows
\begin{equation*}
  u_k(x,t):=u(x+\bar{x}^k,t+\bar{t}_k),
\end{equation*}
then applying Arzel\`{a}-Ascoli theorem to deduce that there exists a subsequence of $\{u_k\}$ (still denoted by $\{u_k\}$) such that
\begin{equation}\label{Main13}
  u_k(x,t)\rightarrow u_\infty(x,t) \,\,\mbox{in}\,\, B_{2R}(0)\times(-R^2,-r^2), \,\,\mbox{as}\,\, k\rightarrow\infty
\end{equation}
for a fixed radius $R>\max\{2r,\,\lambda_0\}$ to be determined later.
Combining \eqref{Main12} with \eqref{Main13} and the definition of $\bar{w}_k$, we derive
\begin{eqnarray*}
  0 \leftarrow \bar{w}_k(x,t)&=&w_k(x+\bar{x}^k,t+\bar{t}_k) \\
  &=&w_{\lambda_0}(x+\bar{x}^k,t+\bar{t}_k)+\varepsilon_k\eta_k(x+\bar{x}^k,t+\bar{t}_k)\\
   &=& u(x+\bar{x}^k,t+\bar{t}_k)-u_{\lambda_0}(x+\bar{x}^k,t+\bar{t}_k)\\
   &=&u_k(x,t)-(u_k)_{\lambda_0}(x,t)\\
   &\rightarrow& u_\infty(x,t)-(u_\infty)_{\lambda_0}(x,t), \,\,\mbox{in}\,\, \left(B_{R}(0)\setminus B_r(0)\right)\times(-R^2,-r^2), \,\,\mbox{as}\,\, k\rightarrow\infty,
\end{eqnarray*}
which implies that
\begin{equation*}
  u_\infty(x,t)-(u_\infty)_{\lambda_0}(x,t)\equiv0, \,\,\mbox{in}\,\, \left(B_{R}(0)\setminus B_r(0)\right)\times(-R^2,-r^2).
\end{equation*}
Hence, it follows that
\begin{equation}\label{Main14}
  u_\infty(x',x_n,t)=u_\infty(x',x_n+\lambda_0,t)=u_\infty(x',x_n+2\lambda_0,t)=...=u_\infty(x',x_n+i\lambda_0,t)
\end{equation}
for any fixed $(x,t)\in\left(B_{R}(0)\setminus B_r(0)\right)\times(-R^2,-r^2)$ with $|x'|>r$, and $i\in \mathbb{N}$ such that
\begin{equation*}
  (x',x_n+(i-1)\lambda_0)\in B_{R}(0)\setminus B_r(0) \,\,\mbox{and}\,\,(x',x_n+i\lambda_0)\not\in B_{R}(0)\setminus B_r(0).
\end{equation*}
However, in terms of the uniform convergence condition on $u(x,t)$ and $\bar{x}^k_n$ is bounded due to $\bar{x}^k\in B_1(x^k)$ and $x^k_n\in [-a,a]$, we choose the radius $R$ large enough such that
\begin{equation*}
  u_\infty(x,t)\leq \delta-1 \,\,\mbox{for}\,\, x_n\leq-\frac{R}{2},\,\,\mbox{and}\,\, u_\infty(x,t)\geq1-\delta \,\,\mbox{for}\,\, x_n\geq\frac{R}{2},
\end{equation*}
which contradicts the equality \eqref{Main14} by selecting
\begin{equation*}
 (x,t)\in\left(B_{R}(0)\setminus B_r(0)\right)\times(-R^2,-r^2)\,\,\mbox{with} \,\, |x'|>r \,\, \mbox{and}\,\, x_n\leq -\frac{R}{2},
\end{equation*}
as illustrated in Figure 5 below.
\begin{center}
\begin{tikzpicture}[scale=0.8]
\draw  [->,purple, very thick] (0.9,1.6)--(1.3,2) node [above, purple, font=\fontsize{12}{12}\selectfont] {$u_\infty\geq1-\delta$};
\draw  [->,purple, very thick] (0.9,-0.9)--(1.3,-1.3) node [below, purple, font=\fontsize{12}{12}\selectfont] {$u_\infty\leq\delta-1$};
 \draw (0,0) ellipse[x radius=3cm, y radius=3cm];
\draw (0,0) ellipse[x radius=1.5cm, y radius=1.5cm];
\draw (0,0) ellipse[x radius=0.5cm, y radius=0.5cm];
\path (0.4,3.2) node[ font=\fontsize{10}{10}\selectfont] {$2R$};
\path (0.3,1.7) node[ font=\fontsize{10}{10}\selectfont] {$R$};
\path (0.2,0.6) node[ font=\fontsize{10}{10}\selectfont] {$r$};
\draw [very thick]  [black] [->,very thick](-4,0)--(4,0) node [anchor=north west] {$x'$};
\draw [very thick]  [black!80][->,very thick] (0,-4)--(0,4) node [black][ above] {$x_n$};
\draw [very thick] [dashed] [blue] (-4,-0.75)--(4,-0.75)node [blue][ font=\fontsize{10}{10}\selectfont][below] {$x_n=-\frac{R}{2}$};
\path (0.9,-0.9)[very thick,fill=red]  circle(1.6pt);
\path (0.9,-0.4)[very thick,fill=red]  circle(1.6pt);
\path (0.9,0.1)[very thick,fill=red]  circle(1.6pt);
\path (0.9,0.6)[very thick,fill=red]  circle(1.6pt);
\path (0.9,1.1)[very thick,fill=red]  circle(1.6pt);
\path (0.9,1.6)[very thick,fill=red]  circle(1.6pt);
\node [below=0.5cm, align=flush center,text width=12cm] at  (0,-4)
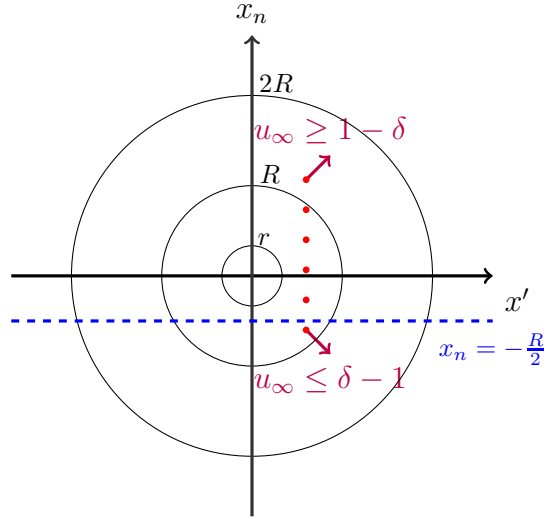
        {Figure 5. The choice of points. };
\end{tikzpicture}
\end{center}
Therefore, we conclude that the assertion \eqref{Main6} holds.

In the sequel, we continue to show that there exists a small positive constant $\varepsilon$ such that
\begin{equation}\label{Main15}
  w_\lambda(x,t)\leq 0, \,\,\mbox{in}\,\, \mathbb{R}^n\times\mathbb{R} \,\, \mbox{for any}\,\, \lambda\in (\lambda_0-\varepsilon,\lambda_0]
\end{equation}
in the case of $\lambda_0>0$.
 Note that combining the aforementioned conclusion \eqref{Main6} and the continuity of $w_\lambda$ with respect to $\lambda$, there exists a small positive constant $\varepsilon$ such that
\begin{equation}\label{Main16}
  \sup_{(x,t)\in \left(\mathbb{R}^{n-1}\times[-a,a]\right)\times \mathbb{R}} w_{\lambda}(x,t)\leq0 \,\, \mbox{for any}\,\, \lambda\in (\lambda_0-\varepsilon,\lambda_0].
\end{equation}
Then in order to verify the validity of \eqref{Main15}, it suffices to prove that
\begin{equation*}\label{Main17}
  \sup_{(x',t)\in \mathbb{R}^{n-1}\times \mathbb{R}, \,|x_n|>a} w_{\lambda}(x,t)\leq0 \,\, \mbox{for any}\,\, \lambda\in (\lambda_0-\varepsilon,\lambda_0].
\end{equation*}
Otherwise, there exists some $\lambda\in(\lambda_0-\varepsilon,\lambda_0]$ such that
\begin{equation}\label{Main18}
  \sup_{(x',t)\in \mathbb{R}^{n-1}\times \mathbb{R}, \,|x_n|>a} w_{\lambda}(x,t)=:A>0.
\end{equation}
If we directly choose $\Omega=\{(x,t)\in \mathbb{R}^{n}\times \mathbb{R}\mid |x_n|>a\}$ as an unbounded set to use the maximum principle established in Theorem \ref{MPUB}\,, the ``size'' of $\Omega^c$ is ``too small" as compared to the ``size'' of $\Omega$ in the sense of limit condition \eqref{MPUB1}, then this condition is not valid for such $\Omega$. Hence, we further need
to shrink the unbounded set $\Omega$.
\begin{center}
\begin{tikzpicture}[scale=0.8]
 \draw[blue!30,fill=blue!30] (-4,3) rectangle (4,1.5);
  \draw[blue!30,fill=blue!30] (-4,-3) rectangle (4,-1.5);
\draw [very thick]  [black] [->,very thick](-4,0)--(4,0) node [anchor=north west] {$x'$};
\draw [very thick]  [black!80][->,very thick] (0,-4)--(0,4) node [black][ above] {$x_n$};
\path node at (-0.3,-0.3) {$0$};
\draw [very thick] [dashed] [blue] (-4,3)--(4,3)node [black][right] {$x_n=M$};
\draw [very thick] [dashed] [blue] (-4,1.5)--(4,1.5)node [black][right] {$x_n=a$};
\draw [very thick] [dashed] [blue] (-4,-3)--(4,-3)node [black][right] {$x_n=-M$};
\draw [very thick] [dashed] [blue] (-4,-1.5)--(4,-1.5)node [black][right] {$x_n=-a$};
\path node at (1.5,2.25) {$\Omega$};
\path node at (1.5,-2.25) {$\Omega$};
\node [below=0.5cm, align=flush center,text width=12cm] at  (0,-4)
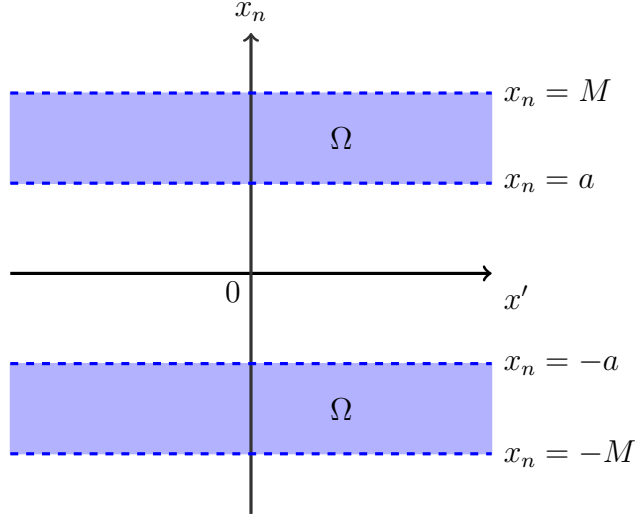
        {Figure 6. The stripe region $\Omega$. };
\end{tikzpicture}
\end{center}
More precisely, employing the uniform convergence
condition of $u(x,t)$, we can select
a sufficiently large constant $M>a$ such that
\begin{equation}\label{Main19}
w_\lambda(x,t)\leq \frac{A}{2} \,\,\mbox{for any} \,\, (x,t)\in \mathbb{R}^{n}\times\mathbb{R}\,\, \mbox{with}\,\, |x_n|\geq M.
\end{equation}
Let
$$\Omega:=\{(x,t)\in \mathbb{R}^{n}\times \mathbb{R}\mid a<|x_n|<M\}$$
be the blue stripe-shaped region in Figure 6,
and we denote
\begin{equation*}
  v_{\lambda}(x,t):=w_{\lambda}(x,t)-\frac{A}{2},
\end{equation*}
then a combination of \eqref{Main16} and \eqref{Main19} yields the exterior condition
\begin{equation*}
  v_{\lambda}(x,t)\leq 0, \,\,\mbox{in}\,\, \Omega^c\times\mathbb{R}.
\end{equation*}
Now we prove the following differential inequality
\begin{equation}\label{Main20}
  (\partial_t-\Delta)^sv_\lambda(x,t)=(\partial_t-\Delta)^s w_\lambda(x,t)=f(t,u(x,t))-f(t,u_\lambda(x,t))\leq0
\end{equation}
is fulfilled at the points in $\Omega\times\mathbb{R}$ where $v_{\lambda}(x,t)>0$.

Since $v_{\lambda}(x,t)>0$ infers that $w_{\lambda}(x,t)>0$, if $a<x_n<M$, then it follows from \eqref{Main2} that
$$1\geq u(x,t)>u_{\lambda}(x,t)\geq1-\delta$$
for any $(x',t)\in\mathbb{R}^{n-1}\times\mathbb{R}$ at the points where $v_{\lambda}(x,t)>0$. Applying the non-increasing assumption on $f(t,u)$ with respect to $u$ in $[1-\delta,1]$, we obtain
\begin{equation*}
  (\partial_t-\Delta)^sv_\lambda(x,t)=f(t,u(x,t))-f(t,u_\lambda(x,t))\leq0
\end{equation*}
at the points in $\Omega\times\mathbb{R}$ with $a<x_n<M$ where $v_{\lambda}(x,t)>0$. While if $-M<x_n<-a$, then \eqref{Main2} implies that
$$-1\leq u_{\lambda}(x,t)<u(x,t)\leq-1+\delta$$
for any $(x',t)\in\mathbb{R}^{n-1}\times\mathbb{R}$ at the points where $v_{\lambda}(x,t)>0$. Using the non-increasing assumption on $f(t,u)$ for $u\in[-1,-1+\delta]$, we derive
\begin{equation*}
  (\partial_t-\Delta)^sv_\lambda(x,t)=f(t,u(x,t))-f(t,u_\lambda(x,t))\leq0
\end{equation*}
at the points in $\Omega\times\mathbb{R}$ with $-M<x_n<-a$ where $v_{\lambda}(x,t)>0$.
In conclusion, we verify that the differential inequality \eqref{Main20} is valid. Applying Theorem \ref{MPUB} to $v_\lambda$ and combining with the definition of $v_\lambda$, we deduce that
\begin{equation*}
w_\lambda(x,t)\leq\frac{A}{2}, \, (x,t) \in \mathbb{R}^n\times\mathbb{R},
\end{equation*}
which is a contradiction with \eqref{Main18}. It follows that the assertion \eqref{Main15} is true, which contradicts the definition of $\lambda_0$, and hence $\lambda_0=0$ and $u(x,t)$ is increasing with respect to $x_n$ for any $t\in\mathbb{R}$.

We end up this step by further proving that $u(x,t)$ is strictly increasing with respect to $x_n$, i.e.,
\begin{equation}\label{Main21}
  w_\lambda(x,t)< 0, \,\,\mbox{in}\,\, \mathbb{R}^n\times\mathbb{R} \,\, \mbox{for any}\,\, \lambda>0.
\end{equation}
If \eqref{Main21} is violated, then there exist a fixed $\lambda_0>0$ and a point $(x^0,t_0)\in \mathbb{R}^n\times\mathbb{R}$ such that $w_{\lambda_0}(x^0,t_0)=0$, and $(x^0,t_0)$ is a maximum point of $ w_{\lambda_0}(x,t)$ in $\mathbb{R}^n\times\mathbb{R}$. On one hand, we directly calculate
\begin{equation*}
  (\partial_t-\Delta)^sw_{\lambda_0}(x^0,t_0) =C_{n,s}\int_{-\infty}^{t_0}\int_{\mathbb{R}^n}
  \frac{-w_{\lambda_0}(y,\tau)}{(t_0-\tau)^{\frac{n}{2}+1+s}}e^{-\frac{|x^0-y|^2}{4(t_0-\tau)}}\operatorname{d}\!y\operatorname{d}\!\tau \geq 0.
\end{equation*}
On the other hand, we have
\begin{equation*}
 (\partial_t-\Delta)^sw_{\lambda_0}(x^0,t_0) =f(t_0,u(x^0,t_0))-f(t_0,u_{\lambda_0}(x^0,t_0))=0.
\end{equation*}
As a consequence of the above estimates, we derive
\begin{equation*}
  w_{\lambda_0}(x,t)\equiv 0, \,\,\mbox{in}\,\, \mathbb{R}^n\times(-\infty, t_0),
\end{equation*}
which contradicts the uniform convergence condition condition of
$u(x,t)$ with respect to $x_n$ for any fixed $(x',t)\in\mathbb{R}^{n-1}\times(-\infty, t_0)$.
Therefore, we deduce that \eqref{Main21} is valid, and then $u(x,t)$ is strictly increasing
with respect to $x_n$.

\noindent \textbf{Step 3.} We finally prove that the entire solution $u(x,t)$ is one-dimensional symmetry for any $t\in\mathbb{R}$, that is,
 $u(x,t)=u(x_n,t)$ is independent of $x'$.
By proceeding similarly as in Step 1 and Step 2, we can derive
$$u(x+\lambda\nu,t)>u(x,t),\,\,  \mbox{in}\,\, \mathbb{R}^n\times\mathbb{R} \,\, \mbox{for any}\,\, \lambda>0,$$
and every vector $\nu=(\nu_1,...,\nu_n)$ with $\nu_n>0$. This implies that $u(x,t)$ is strictly increasing along any direction which has an acute angle with the positive $x_n$ axis.

 Let $\nu_n \rightarrow 0$, then by the continuity of $u(x,t)$, the inequality is still preserved in the sense
$$u(x+\lambda\nu,t)\geq u(x,t),\, (x,t)\in \mathbb{R}^n\times\mathbb{R} \,\, \mbox{for any}\,\, \lambda>0.$$
Note that $\nu$ can be any given direction perpendicular to $x_n$ axis, then we conclude that $u(x,t)$ must be independent of $x'$, i.e., $u(x,t)=u(x_n,t)$ for any $t\in\mathbb{R}$. Hence, we complete the proof of Gibbons' conjecture for master equation.
\end{proof}

\section*{Acknowledgments}
The work of the first author is partially supported by MPS Simons foundation 847690,
and the work of the second author is partially supported by the National Natural Science Foundation of China (NSFC Grant No.12101452).

\end{document}